\documentclass[12pt]{article}
\usepackage[]{times}
\usepackage{amsmath}
\usepackage{paralist}
\usepackage{fullpage}
\usepackage{amssymb}
\usepackage{mathrsfs}
\usepackage{bm}
\usepackage{graphicx}
\usepackage{graphicx}
\usepackage{wrapfig}
\usepackage{epsfig}
\usepackage{url}
\usepackage{fullpage}
\usepackage{psfrag}
\usepackage{verbatim}
\usepackage{amsthm}
\usepackage[]{times}
\usepackage{amsmath}
\usepackage{paralist}
\usepackage{fullpage}
\usepackage{amssymb}
\usepackage{amsrefs}

\usepackage[colorlinks=true, pdfstartview=FitV, linkcolor=blue,
            citecolor=blue, urlcolor=blue]{hyperref}
\usepackage[usenames]{color}
\definecolor{Red}{rgb}{0.7,0,0.1}
\definecolor{Green}{rgb}{0,0.6,0}

\newcommand{\nboldsymbol}[1]{\boldsymbol{#1}}

\def\hu{\hat{\mathbf{u}}}
\def\hU{\hat{\mathbf{U}}}
\def\exi{\mathbf{e}_{\nboldsymbol{\xi}}}

\newcommand{\rmd}{\mathrm{d}}           % derivatives
\newcommand{\rme}{\mathrm{e}}           % Euler constant
\usepackage{mathrsfs}
\usepackage{bm}
\usepackage{graphicx}
\usepackage{graphicx}
\usepackage{wrapfig}
\usepackage{epsfig}
\usepackage{url}
\usepackage{fullpage}
\usepackage{psfrag}
\usepackage{verbatim}

\newcommand{\R}{\mathbb{R}}

\renewcommand{\geq}{\geqslant}
\renewcommand{\leq}{\leqslant}

\newcommand{\zetas}{\tilde{\zeta}}

\numberwithin{equation}{section}
\newtheorem{thm}{Theorem}[section]
\newtheorem{lemma}{Lemma}[section]
\newtheorem{cor}{Corollary}[section]
\newtheorem{prop}{Proposition}[section]

\newtheorem{remark}{Remark}[section]
\newtheorem{definition}{Definition}[section]

\DeclareMathOperator{\Li}{Li}

\title{Symmetry Breaking and Uniqueness for the Incompressible Navier-Stokes Equations}
\author{Radu Dascaliuc \thanks{Department of Mathematics,  Oregon State University, Corvallis, OR, 97331.}
\and Nicholas Michalowski \thanks{Department of Mathematics,  New
    Mexico State University, Las Cruces, NM,  88003.}
\and Enrique Thomann\thanks{Department of Mathematics,  Oregon State University, Corvallis, OR, 97331. }
\and
Edward C. Waymire\thanks{Department of Mathematics,  Oregon State University, Corvallis, OR, 97331. Corresponding author. \tt {waymire@math.oregonstate.edu}.}
}

\date{}

\begin{document}

\maketitle

\begin{abstract}
The present article establishes connections between the structure
of the deterministic Navier-Stokes equations and the structure of (similarity) equations that
govern self-similar solutions as
expected values of certain naturally associated
stochastic cascades.  
A principle result is that  explosion criteria
for the stochastic cascades involved in the probabilistic 
representations of solutions to the respective equations coincide. 
  While the uniqueness problem itself remains unresolved, these
connections provide interesting problems and 
possible methods for investigating
symmetry breaking and the uniqueness problem for Navier-Stokes equations.
In particular, new branching Markov chains, including a {\it dilogarithmic branching random walk} on the multiplicative group $(0,\infty)$, naturally
arise as a result of this 
investigation.
\end{abstract}

{\bf 
The role of scaling in the question of uniqueness of mild solutions to 3D Navier-Stokes equations is the central theme of this investigation. 
%In particular, we describe a probabilistic framework for comparing the uniqueness problem for scaling-invariant solutions with the uniqueness for the general solutions appropriate functional setting. 
In particular, we describe a framework, where the uniqueness for both scale-invariant and general problems is re-cast in terms of a non-explosion property of associated stochastic cascades. Thus, if the explosion event of the self-similar cascade is probabilistically different from the explosion event for the general, non-symmetric cascade in appropriate settings, then we would have a manifestation of symmetry breaking in the Navier-Stokes uniqueness problem -- the scaling-invariant case being qualitatively different.  While we are only able to prove partial results related to the associated explosion problems, the main conclusion of this paper is that the self-similar (scaling-invariant) explosion and the general, non-symmetric explosion (in appropriate functional settings) are the same, suggesting that scaling symmetry may be directly involved in the eventual solution of the outstanding Navier-Stokes well-posedness problem. We note that the idea of employing scaling-invariant solutions in the context of well-posedness goes back to Leray \cite{JL1934}, while the idea to use stochastic cascades to prove existence of mild solutions is due to Le Jan and Sznitman  \cite{YLAS1997}.
 }

\section{Introduction}

\subsection{Navier-Stokes equations and their scaling properties.}

The physics of unrestricted three-dimensional incompressible fluid flow is mathematically encoded 
in the corresponding set of Navier-Stokes equations (NSE) governing the time 
evolution of velocity (momentum)
 ${\mathbf u}$ and pressure $p$ in three dimensional Euclidean space.
Letting 
${\mathbf u}({\mathbf x},t)$ denote the velocity of an incompressible fluid at the position ${\mathbf x}\in \mathbb{R}^3$ and
and time $t\ge 0$, essentially Newton's law of motion may
be cast as
\begin{equation}
\label{NSeqn}
\frac{\partial {\mathbf u}}{\partial t}   +   {\mathbf u}\cdot\nabla{\mathbf u} =   \nu\Delta{\mathbf u} -\nabla p + \mathbf{g},
\quad \nabla\cdot {\mathbf u} = 0, \  \mathbf{u}({\mathbf x},0^+) = \mathbf{u}_0({\mathbf x}), \ {\mathbf x}\in\mathbb{R}^3,\  t > 0, 
\end{equation}
where $\nu > 0$ is a positive (viscosity) parameter,  $\nabla = ({\partial/\partial x_j})_{1\le j\le 3},$
$\Delta = \nabla\cdot\nabla$ is the (vector) {\it Laplacian} operator, and $\mathbf{g}$ is an (external
forcing) function with values in ${\mathbb R}^3$.  More generally (\ref{NSeqn}) may be posed on a  
domain in ${\mathbb R}^3$ with boundary.  However for convenience in this paper we will consider the 
free-space model without boundary or an external force. 
%or, at most, a bounded domain with periodic boundary conditions (i.e., torus).

%The fluid mass density is assumed constant and normalized to unity for convenience without further comment. 
The term ${\partial {\mathbf u}/\partial t}   +   {\mathbf u}\cdot\nabla{\mathbf u}$ represents the acceleration of a
fluid parcel within a Lagrangian reference frame.  In particular,
the non-linearity ${\mathbf u}\cdot\nabla{\mathbf u}$ is intrinsic to this
 description of the flow and cannot be eliminated.  The 
viscous force $\nu\Delta\mathbf{u}$ is the result of a linearization
of stress-strain forces between fluid parcels composing the
fluid, and the {\it divergence-free}
condition $\nabla\cdot\mathbf{u} =0$ provides 
conservation of mass; also referred to as {\it incompressibility}.  The pressure 
gradient term $\nabla p$ is a fourth unknown in the set of 
four equations describing the $n=3$ coordinates  of velocity 
$\mathbf{u} = (u_1,u_2,u_3)$ and
the scalar pressure $p$. We refer to \cite{PL2002} and \cite{RT2001} for more background on the physical derivation of the Navier-Stokes equations.   

The unique determination of 
${\mathbf u}$ from the given viscosity parameter $\nu > 0$, external forcing $\mathbf{g}$ (in our case $\mathbf{g}=\mathbf{0}$), and initial data $\mathbf{u}_0$ 
is an obvious question for 
both the physics and the mathematics of fluid flow.
After more than one-hundred years of research
it remains unknown whether smooth initial data ${\mathbf u}_0$ leads to the existence of 
unique smooth (regular) solutions, valid for all time.  
%A breakdown could indeed occur as a manifestation of {\it turbulence}.   
It is believed that mathematical progress on this issue is closely connected to understanding the physical phenomenon of {\em turbulence}.
As a consequence, the resolution of the uniqueness and regularity problem for the Navier-Stokes equations ranks
among the most important open problems of contemporary applied and theoretical mathematics.   

The current state of the regularity issue may be viewed through the prism of natural scaling (symmetry) peculiar to of the Navier-Stokes equations as follows

\begin{equation}\label{scaling}\begin{array}{c} 
\mbox{If ${\mathbf u}(\mathbf{x},t), p(\mathbf{x},t)$ is a solution to (\ref{NSeqn}), then for any
scaling parameter $r > 0$,}  \\
\mbox{$\mathbf{u}_r(\mathbf{x},t)
 =r\mathbf{u}(r\mathbf{x},r^2 t), p_r(\mathbf{x},t) = r^2 p(r\mathbf{x},r^2t)$ is also a solution}\\ \mbox{with initial data
$r{\mathbf u}_0(r{\mathbf x}). $}
\end{array}\end{equation}

The quantities of the flow (typically represented by certain norms of $\mathbf{u}$) that preserve this scaling are called {\em critical}, the ones that grow as $r\to 0$ are {\em super-critical}, and the ones that decrease are {\em sub-critical}. For example, since the pioneering work of Leray in the 1930's (see \cite{JL1934}, which still remains a benchmark for  the regularity problem), it is known that NSE possess global-in-time weak solutions that are bounded in $L^2$.  If we re-scale the $L^2$-norm of $\mathbf{u}$ according to the scaling above we obtain
\[\|\mathbf{u}_r\|^2_{2,\infty}:=\sup\limits_{t\in[0,\infty)}\|\mathbf{u_r}(t)\|_2^2=\sup\limits_{t\in[0,\infty)}\int\limits_{\mathbb{R}^3}r^2|\mathbf{u}(r\mathbf{x},r^2 t)|^2\,d\mathbf{x}=\frac{1}{r}\|\mathbf{u}\|^2_{2,\infty},\]
and so $\|\cdot\|_{2,\infty}$ is a super-critical quantity. Yet, according to Leray's result, the solution is regular as long the $L^2$ norm of the (vector) gradient remains bounded:
\[
\|\nabla\mathbf{u}_r\|^2_{2,\infty}:=\sup\limits_{t\in[0,\infty)}\|\nabla\mathbf{u}_r(t)\|_{L^2}^2=
\sup\limits_{t\in[0,\infty)} \int\limits_{\mathbb{R}^3}
r^2\sum\limits_{i,j=1}^3 (\partial_{x_i}u_j(r\mathbf{x}, r^2t))^2\,d\mathbf{x}=r\|\nabla\mathbf{u}\|^2_{2,\infty},
\]
i.e. Leray's regularity condition is sub-critical.

More modern regularity criteria still suffer similar scaling defects:  The Escauriaza, Seregin and \v{S}ver\'ak criterion involving boundedness of the $L^3$-norm (see \cite{LEetAl2003}), as well as the Koch and Tataru condition of smallness of the initial data in the BMO$^{-1}$ functional space (\cite{HKDT2001}), are each critical in nature.

This gap between what is known for the solutions of NSE (all of which are super-critical) and the sufficient conditions for regularity, is one of the manifestations of the important role scaling plays in the NSE well-posedness problem.

There is a growing consensus that functional and harmonic analysis techniques alone would not be sufficient to break the regularity problem by obtaining a super-critical condition for well-posedness. 
Specifically, there are examples on NSE-like systems that blow-up in finite time despite many functional properties characteristic to NSE (\cites{SM2001,NKNP2005,AC2008,TT2014}).

This suggests a necessity of developing new approaches to understand NSE and non-linear systems in general. In particular, our results suggest that the stochastic multiplicative cascade framework introduced by Le Jan and Sznitman (\cite{YLAS1997}) may provide new insights into the NSE regularity problem. Indeed, in this note we establish a new scaling-critical condition for uniqueness of solutions, as well as provide evidence of a connection between the issue of uniqueness and natural scaling of the NSE.

A natural way to explore the role of scaling in the theory of NSE is to consider {\em scaling-invariant} or {\em self-similar} solutions, i.e. the solutions satisfying

\begin{equation}\label{LerayScaling}{\mathbf u}_r = {\mathbf u},
\ p_r = p, \quad \forall r > 0.\end{equation}

Leray \cite{JL1934} observed that if ${\mathbf u},p$ is a self-similar solution to (\ref{NSeqn}),
then upon choosing $r\equiv r(t) = 1/\sqrt{t}$ for fixed $t > 0$, one has
$$\mathbf{u}(\mathbf{x},t) = \frac{1}{\sqrt{t}}
\mathbf{u}(\frac{\mathbf{x}}{\sqrt{t}},1) = \frac{1}{\sqrt{t}}
{\mathbf U}(\frac{\mathbf{x}}{\sqrt{t}}),$$
where
\begin{equation}
\label{Lerayeqtn}
-\Delta \mathbf{U} - \frac{1}{2} \mathbf{U} - \frac{1}{2}(X\cdot\nabla) \mathbf{U} + (\mathbf{U} \cdot \nabla) \mathbf{U} = - \nabla P,
~~~~
\nabla\cdot \mathbf{U} =  0.
\end{equation}
Leray himself had the idea to use this self-similarity (backwards in time, with $r(t)=1/\sqrt{T-t}$\,) to produce an example of blow-up in the NSE problem. This was eventually proved impossible due to the work of Tsai (\cite{TPT1998}), as well as Ne\v{c}as, R{\r{u}}{\v{z}}i{\v{c}}ka, and \v{S}ver\'ak (\cite{JNetAl1996}) (see also \cite{PL2002})  who established that backward-in-time the only self-similar solution is $\mathbf{0}$. Study of forward in time self-similar solutions, particularly of equation (\ref{Lerayeqtn}), revealed several important existence and uniqueness as well as regularity results (\cites{YGTM1989,MCFP1996,YM1997,ZG2006}). In particular, Meyer  (\cite{YM1997}) provided a framework of constructing solutions  that are unique in a weak $L^3$-space starting from ``small'' initial data. We note that the self-similar solutions must invariably posses singularity at the origin, as they are homogeneous functions of degree -1. The weak-$L^3$ space is a natural functional space for such functions. Later Gruji\'c (\cite{ZG2006}) showed that the solutions built  by Meyer are in fact smooth (outside the origin). More recently, Jia and \v{S}ver\'ak  (\cites{HJVS2014,HJVS2013}) proved existence of smooth solutions for (\ref{Lerayeqtn}), without a smallness assumption, pointing to a potential for lack of uniqueness of self-similar solutions for `large' initial data and showing a pathway of how such solutions might be used to produce blow-up in Navier-Stokes equations. Cannone and Karch (\cite{MCGK2004}) also argued for the connections between the theory of self-similar solutions with large initial data and possible emergence of singularities in the NSE.

The fact that self-similar solutions could be used to prove/disprove well-posedness for general NSE is another manifestation of the particular importance of scaling symmetries in the Navier-Stokes equations.

\subsection{The question of symmetry breaking -- the description of the main results.}

In this paper we seek to provide an approach to both self-similar, as well as general NSE problems that could shed light into the specific issue related to this natural scaling, to be referred to as {\em symmetry breaking}, namely:

\vspace{.2in}

\noindent{\em Is the uniqueness of solutions to NSE tied to the uniqueness of self-similar solutions?}

\vspace{.2in}

If the solutions to (\ref{Lerayeqtn}) are unique, yet the solutions to (\ref{NSeqn}) are not, then we have a manifestation of symmetry breaking in NSE, signaling that a possible lack of well-posedness could be the result  of a mechanism that magnifies/creates deviations from natural scaling present in the initial data.
On the other hand, if the uniqueness for (\ref{Lerayeqtn}) is closely tied to the uniqueness in (\ref{NSeqn}), then the well-posedness problem is essentially connected to the natural symmetries of Navier-Stokes equations.
Thus, the notions of scaling invariance and self-similarity, considered from the perspective of symmetry breaking, provide
the central focus of the present paper.

We consider this issue in the framework of Le Jan-Sznitman stochastic multiplicative cascades, developed in \cite{YLAS1997}, combined with the idea of majorizing kernels, introduced in \cite{RBEtAl2003}, to investigate existence and uniqueness of a mild solutions to the NSE (see (\ref{mildNSE}) below). In this framework, a multiplicative cascade process is associated to the mild formulation of NSE in Fourier space, and a solution is recovered form the initial data via an expected value of a certain recursive product along a generated tree. The space of initial data allowed is in part governed by the choice of the majorizing kernel (see Section \ref{sec2} for details).

In order to guarantee finiteness of the tree, a {\em thinning procedure} is usually employed. The thinning, which involves a chance of artificially terminating a branch, is guaranteed to generate a finite cascade,  producing a unique mild solution to NSE,  but at an expense of shrinking the smallness condition on the initial data.

In contrast to the classical Le Jan-Sznitman approach, we will not employ a thinning procedure to terminate the cascade (see Section \ref{sec2}). Elimination of thinning is a step towards accommodating wider families of initial data by  relaxing, and eventually removing the aforementioned smallness condition. However, in the absence of thinning, one has to deal with a possibility of the formation of infinite cascade trees in finite time -- the phenomenon called {\em explosion}, and our main object of study.

In particular, we will not be concerned with the issue of existence of mild solutions built with such procedure (the solution is guaranteed to exist as long as the cascade is non-exploding and the associated expected values are finite -- see Section \ref{sec2}). Also, we will not study regularity properties of such solutions (a difficult question, especially for more general spaces of initial data). Instead, our goal is to show that the {\em explosion} phenomenon in such cascades can be used as a {\em surrogate for uniqueness} for the solutions of the NSE in a certain functional class, allowing us to classify the associated uniqueness problems by the corresponding explosion time random variables.

Specifically, we use this approach to study two families of NSE initial data: one governed by the Bessel majorizing kernel (\ref{bessel_kernel}) (in which case we are able prove the non-explosion), as well as the dilogarithmic kernel (\ref{dilog_kernel}) (which allows for a much wider space of initial data, but with a more nuanced explosion problem). 

We also adapt this approach to the study of scaling-invariant solutions (see Section \ref{sec3}). This forces a very different choice of the scaling parameter $r$ (see (\ref{self-sim_scaling})) than the one in (\ref{LerayScaling}), partly because the problem is posed in the Fourier setting. Nevertheless we show that the self-similar mild formulation  we use -- (\ref{mildssns}), and the Leray equation (\ref{Lerayeqtn}) are in fact equivalent (Proposition \ref{prop:Leray_eq}). Moreover, although the resulting cascade is quite different from the general NSE case described in Section \ref{sec2}, the dilogarithmic density appears naturally in the context of scaling-invariance (see (\ref{mildssns})).

The explosion problems themselves are defined in terms of the {\em explosion time} random variables in both non-symmetric and scaling-invariant cases (see Definitions \ref{expltime} and \ref{selfsimilarexplosion}) -- {\em critical} (with respect to the scaling) quantities. Using the Le Jan-Sznitman martingale argument, we show that in the case of general NSE, the non-explosion of the associated multiplicative cascade provides a {\em scaling-critical sufficient condition for  uniqueness} -- see Proposition \ref{prop:uniq} and Remark \ref{crit_cond}.

A natural question is to compare the case of dilogarithmic majorizing kernel in general, non-symmetric setting to the scaling-invariant case.  While we were unable to fully resolve the associated explosion problems, the main conclusion  of this analysis -- see Theorem \ref{thm:main} -- is that at the level of cascades, 

\vspace{.2in}

\noindent{\em The explosion problem is  the same in both self-similar and general case (in dilogarithmic settings).} 

\vspace{.2in}

This result provides evidence for a lack of symmetry breaking in the Navier-Stokes problem.

\medskip

%\subsection{The plan for the paper.}
The rest of the paper is organized as follows.

In Section \ref{sec2} we will use the Le Jan-Sznitman cascade without thinning, together with the idea of majorizing kernels, to formulate an explosion problem (Definition \ref{expltime}) closely connected to the issue of uniqueness of solutions to the mild NSE formulation (\ref{mildNSE}) -- see Proposition \ref{prop:uniq}. Two particular kernels are considered (see (\ref{bessel_kernel}) and (\ref{dilog_kernel})\,). In the case of Bessel kernel, $h_b$, we can prove the non-explosion (see Theorem \ref{th:nonexplosionBessel} in the appendix). In the case of less restrictive dilogarithmic kernel, $h_d$, we prove that the explosion is related to the uniqueness property of a certain solution to a non-linear PDE (see Proposition \ref{prop:PsiDE}).

In Section \ref{sec3}, an analogous procedure will be employed to arrive to an explosion problem in the self-similar case (Definition \ref{selfsimilarexplosion}). In particular, we relate the solutions obtained with this method to the solutions of the above-mentioned Leray equation (Proposition \ref{prop:Leray_eq}), and prove (Theorem \ref{thm:0-1}) that the explosion itself is a zero-one event (i.e. it is essentially deterministic). We also show in Theorem \ref{thm:main} that this explosion event has the same distribution as the explosion in the dilogarithmic case described in Section \ref{sec2} -- the evidence towards similarity between the two uniqueness problems.

Section \ref{sec4} is devoted to comparison of the general and self-similar cases from the point of view of the symmetry breaking question, and discusses some open problems.

Finally, Section \ref{appendix} is the Appendix containing the proofs of the technical results related to the Bessel and dilogarithmic random walks, which appear in the functional settings adopted in Section \ref{sec2}. It is worth noting here that {\em dilogarithmic random walk} (which arises naturally in the context of this paper)
appears to be a new multiplicative stochastic process that may be of broader interest, e.g,  see \cites{ML1995, ML1997, AK1995} for other occurrences of the dilogarithmic distribution in physics.

\section{Navier-Stokes Cascades \& An Explosion Problem}\label{sec2}

Next, we will describe the mathematical framework we use to 
analyze the existence and uniqueness problem for the Navier-Stokes equations more precisely, by specifying the meaning of \lq\lq solution\rq\rq.  
Due to lack of existence results for smooth (classical) solutions, the notion of solution in the {\it weak sense} is frequently used, where derivatives are in the distributional sense, as this allows one to search among  functions $\mathbf{u}$ that are locally square-integrable in space. Namely, a divergence-free vector field $\mathbf{u}(t)$ is called a weak solution if for all $\mathbf{v}:\mathbb{R}^3\to\mathbb{R}^3$ -- smooth, divergence-free functions with compact support, 
\[\langle \mathbf{u}(t),\mathbf{v}\rangle-\langle \mathbf{u}(0),\mathbf{v}\rangle=\int\limits_0^t
\left(\,\nu\langle \mathbf{u}(s),\Delta \mathbf{v}\rangle+\langle \mathbf{u}(s),(\mathbf{u}(s)\cdot\nabla)\, \mathbf{v}\rangle\,\right)\,ds\;,\] 
under the implicit assumptions on $\mathbf{u}$ that make the integrals above valid. Here, $\langle \mathbf{u},\mathbf{v}\rangle=\int_{\mathbb{R}^3} \mathbf{u}\cdot\overline{\mathbf{v}}\,d\mathbf{x}$ is the (complex) $L^2$ inner product. This definition was introduced by Leray to provide a mathematical framework that would accommodate the
possibility that velocities may not be smooth at some \lq\lq small set\rq\rq of points where 
``turbulence'' is present. Indeed, Leray's approach is proven to produce solutions that are are smooth except possibly a singular set of one-dimensional Hausdorff measure zero (see \cite{LC1982}).

Taking Fourier transform in $\mathbf{x}$ in the equation above, we notice that
\[
\langle \hat{\mathbf{u}}(t),\hat{\mathbf{v}}\rangle-\langle \hat{\mathbf{u}}(0),\hat{\mathbf{v}}\rangle=\int\limits_0^t
\left(-\nu\langle |\nboldsymbol{\xi}|^2\hat{\mathbf{u}}(s), \hat{\mathbf{v}}\rangle+\langle \hat{\mathbf{u}}(s),{ \mathcal{F}\left[ (\mathbf{u}(s)\cdot\nabla)\mathbf{v}\right]}
\rangle\right)\,ds,
\]
where $\hat{w}(\nboldsymbol{\xi})=\mathcal{F}[w](\nboldsymbol{\xi})=(2\pi)^{-3/2}\int_{\mathbb{R}^3} w(\mathbf{x})\ e^{-i\mathbf{x}\cdot\nboldsymbol{\xi}}d\nboldsymbol{\xi}$ is the the Fourier transform of $w$.

Using the reality condition $\hat{\mathbf{u}}(-\nboldsymbol{\xi})=\overline{\hat{\mathbf{u}}}(\nboldsymbol{\xi})$, the last term can be written as:
 \[\langle \hat{\mathbf{u}}(s),{\mathcal{F}{\left[(\mathbf{u}(s)\cdot\nabla)\mathbf{v}\right]}}\rangle=\frac{1}{(2\pi)^{3/2}}\langle \hat{\mathbf{u}}(s),\hat{u}_k(s)*(i\xi_k \hat{\mathbf{v}})\rangle=\frac{-i}{(2\pi)^{3/2}}\langle \xi_k\hat{u}_k(s)*\hat{\mathbf{u}}(s), \hat{\mathbf{v}}\rangle,\]
where $v*w(\nboldsymbol{\xi})=\int_{\mathbb{R}^3}v(\nboldsymbol{\xi}-\nboldsymbol{\eta})w(\nboldsymbol{\eta})\,d\nboldsymbol{\eta}$  is the convolution of functions $v$ and $w$. Moreover, to incorporate divergence-free property in the first term of the inner product above, we can write
\[
\langle\frac{-i}{(2\pi)^{3/2}} \xi_k\hat{u}_k(s)*\hat{\mathbf{u}}(s),\mathbf{v}\rangle=\langle \frac{|\nboldsymbol{\xi}|}{(2\pi)^{3/2}}\int\limits_{\mathbb{R}^3}\hat{\mathbf{u}}(\nboldsymbol{\eta},s)\odot_{\nboldsymbol{\xi}} \hat{\mathbf{u}}(\nboldsymbol{\xi}-\nboldsymbol{\eta},s)\, d\nboldsymbol{\eta},\mathbf{v}\rangle, \]
with
\begin{equation}
\label{algebraicop0}
\hat{\mathbf{v}}(\nboldsymbol{\eta_1})\odot_{\nboldsymbol{\xi}} \hat{\mathbf{w}}(\nboldsymbol{\eta_2})=-i(e_{\nboldsymbol{\xi}}\cdot \hat{\mathbf{w}}(\nboldsymbol{\eta_2}))\pi_{{\nboldsymbol{\xi}}^\perp} \hat{\mathbf{v}}(\nboldsymbol{\eta_1}),  
\end{equation}
where ${e}_{\nboldsymbol{\xi}} = {\nboldsymbol{\xi}/|\nboldsymbol{\xi}|}$ and $\pi_{{\nboldsymbol{\xi}}^\perp}\mathbf{v} = \mathbf{v}-(e_{\nboldsymbol{\xi}}\cdot \mathbf{v})e_{\nboldsymbol{\xi}}$ is the projection of $\mathbf{v}$ onto the plane orthogonal to $\nboldsymbol{\xi}$.   

Since $|\nboldsymbol{\xi}|\int_{\mathbb{R}^3}\hat{\mathbf{u}}(\nboldsymbol{\eta},s)\odot_{\nboldsymbol{\xi}} \hat{\mathbf{u}}(\nboldsymbol{\xi}-\nboldsymbol{\eta},s)\, d\eta$ is divergence-free, we conclude that weak solutions satisfy

\[
\hat{\mathbf{u}}(t)-\hat{\mathbf{u}}(0)\overset{\makebox[0pt]{\mbox{\normalfont\tiny a.e.}}}{=}\int\limits_0^t\left(  -\nu|\nboldsymbol{\xi}|^2\hat{\mathbf{u}}(s)+\frac{|\nboldsymbol{\xi}|}{(2\pi)^{3/2}}\int\limits_{\mathbb{R}^3}\hat{\mathbf{u}}(\nboldsymbol{\eta},s)\odot_{\nboldsymbol{\xi}} \hat{\mathbf{u}}(\nboldsymbol{\xi}-\nboldsymbol{\eta},s)\, d\nboldsymbol{\eta}\right)\, ds\;,
\]

which leads to the following mild formulation of the Navier-Stokes equations:

\begin{equation}\label{mildNSE}
\hat{\mathbf{u}}(\nboldsymbol{\xi},t)= \hat{\mathbf{u}}(\nboldsymbol{\xi},0)e^{-\nu|\nboldsymbol{\xi}|^2t}+\int\limits_0^te^{-\nu|\nboldsymbol{\xi}|^2s}\frac{|\nboldsymbol{\xi}|}{(2\pi)^{3/2}}\int\limits_{\mathbb{R}^3}\hat{\mathbf{u}}(\nboldsymbol{\eta},t-s)\odot_{\nboldsymbol{\xi}} \hat{\mathbf{u}}(\nboldsymbol{\xi}-\nboldsymbol{\eta},t-s)\, d\nboldsymbol{\eta}\,d s\;.
\end{equation}

We note that weak solutions automatically satisfy the above mild formulation, and the solutions to (\ref{mildNSE}) are weak solutions provided they are (uniformly locally) square integrable in both space and time variables  (\cite{PL2002}).

{The stochastic cascade framework in Fourier space was
 introduced in \cite{YLAS1997} for the analysis of 
(\ref{mildNSE}).
The basic ingredients of  the recursively defined
 stochastic object (cascade) associated with the %an unforced ($\mathbf{g} = \mathbf{0}$) 
problem (\ref{mildNSE}) consists of}
(i) a continuous time binary branching Markov process in three-dimensional Fourier wavenumber 
space, and (ii) an algebraic operation $\odot_{\nboldsymbol{\xi}}$
defined in \eqref{algebraicop0}.
The stochastic process is initiated with a Fourier mode $\mathbf{0}\neq\nboldsymbol{\xi}\in\R^3$, where it holds for an exponentially distributed length
of time $T_{\nboldsymbol{\xi}}$ with intensity $\nu |\nboldsymbol{\xi}|^2$.  Upon expiration of time $T_{\nboldsymbol{\xi}}$, the
particle either dies or splits into a pair of frequencies (modes)
$(\mathbf{W}_1,\mathbf{W}_2)\in\R^3\times\R^3$.  The random events of
either dying or splitting occur with equal probabilities and
independently of $T_{\nboldsymbol{\xi}}$.  In the case of a split the new frequencies are 
subject to the local conservation of frequencies condition
\begin{equation}
\label{conserve}
\mathbf{W}_1 + \mathbf{W}_2 = \nboldsymbol{\xi},
\end{equation}
 and distributed according to
  \begin{equation}
  \label{brwtypes}
  \mathbb{E}f({\mathbf W}_1,{\mathbf W}_2) = \int_{\mathbb{R}^3}f(\nboldsymbol{\eta},\nboldsymbol{\xi}-\nboldsymbol{\eta})\frac{h(\nboldsymbol{\eta})h(\nboldsymbol{\xi}-\nboldsymbol{\eta})}{ h*h(\nboldsymbol{\xi})}\, d\nboldsymbol{\eta},
  \end{equation}
where $h:\R^3\backslash\{\mathbf{0}\}\to (0,\infty)$ is a measurable function with full support
 for which $h*h(\nboldsymbol{\xi}) < \infty$
for each $\nboldsymbol{\xi}\neq \mathbf{0}$, introduced in \cite{RBEtAl2003} as a {\it majorizing kernel} for 
(\ref{NSeqn}).   The two special choices given in \cite{YLAS1997} involved 
\begin{gather}
h_d(\nboldsymbol{\xi}) = \frac{1}{|\nboldsymbol{\xi}|^2},\qquad \nboldsymbol{\xi}\neq \mathbf{0},\label{dilog_kernel}\\
h_b(\nboldsymbol{\xi}) = e^{-|\nboldsymbol{\xi}|}/|\nboldsymbol{\xi}|,\qquad \nboldsymbol{\xi}\neq \mathbf{0}, \label{bessel_kernel}
\end{gather}
 for which one has the identities
\begin{equation}
\label{hstarh}
h*h(\nboldsymbol{\xi}) = c|\nboldsymbol{\xi}|h(\nboldsymbol{\xi}), 
\end{equation}
with $c=c_1 = \pi^3$ if $h=h_d$, and $c=c_2 = {2\pi}$ if $h=h_b$. We will refer to $h_b$ as the {\em Bessel majorizing kernel}, and to $h_d$ as {\em dilogarithmic majoring kernel}. The connection with Bessel and dilogarithmic distributions will be given in Propositions \ref{besselbmc} and \ref{dilogbrw} (see also Remark \ref{rem:dilog}).

We define the conditional densities:
\begin{gather*}
  H_d(\eta\,|\,\xi)=\frac{h_d(\nboldsymbol{\eta})h_d(\nboldsymbol{\xi}-\nboldsymbol{\eta})}{h_d*h_d(\nboldsymbol{\xi})}=\frac{|\nboldsymbol{\xi}|}{\pi^3|\nboldsymbol{\xi}-\nboldsymbol{\eta}|^2|\nboldsymbol{\eta}|^2};\\
  H_b(\eta\,|\,\xi)=\frac{h_b(\nboldsymbol{\eta})h_b(\nboldsymbol{\xi}-\nboldsymbol{\eta})}{h_b*h_b(\nboldsymbol{\xi})}=\frac{e^{|\nboldsymbol{\xi}|}}{2\pi} \frac{e^{-|\nboldsymbol{\eta}|} e^{-|\nboldsymbol{\xi}-\nboldsymbol{\eta}|}}{|\nboldsymbol{\eta}||\nboldsymbol{\xi}-\nboldsymbol{\eta}|}.
\end{gather*}

If the particle dies, the (Fourier transformed) 
forcing term, evaluated at its parent mode and appropriately scaled, is attached
to the terminal node.  Otherwise, if branching occurs this rule is repeated from
each of the nodes at respective frequencies $\mathbf{W}_1$ and $\mathbf{W}_2$.  
The focus of the present article is the unforced ($\mathbf{g} = \mathbf{0}$) equation (\ref{NSeqn}), 
in which case such a convention may be viewed as a \lq\lq thinning\rq\rq operation,
that may not be necessary so long as there are only finitely many branches
by any finite time $t$.
  If thinning is applied, however, then the associated genealogical
tree is that of a critical binary Galton-Watson process and therefore
in fact almost surely finite, e.g., see \cites{KAPN1972, RBEW2009}.

To be clear, in the case of no-forcing (${\mathbf g} = 0$)
the Le Jan-Sznitman algorithm results in a 
{\it thinning} of the full binary tree that may be ignored so long as
the branching process is {\it non-explosive}. This observation will
be elaborated upon as a point of focus in the present paper.

 The algebraic operation $\odot_{\nboldsymbol{\xi}}$ is applied to a vector-valued function
of the offspring $(\mathbf{W}_1,\mathbf{W}_2)\in\mathbb{R}^3\times\mathbb{R}^3$,
provided by the initial data $\hat{u}_0:\mathbb{R}^3\to\mathbb{R}^3$,  
at  each node of the genealogical tree having parental wavenumber $\nboldsymbol{\xi}$ as defined by
\begin{equation}
\label{algebraicop}
\hat{\mathbf{u}}_0(\nboldsymbol{\eta}_1)\odot_{\nboldsymbol{\xi}} \hat{\mathbf{u}}_0(\nboldsymbol{\eta}_2)=-i(\mathbf{e}_{\nboldsymbol{\xi}}\cdot \hat{\mathbf{u}}_0(\nboldsymbol{\eta}_2))\pi_{\nboldsymbol{\xi}^\perp} \hat{\mathbf{u}}_0(\nboldsymbol{\eta}_1),  
\end{equation}
where $\mathbf{e}_{\nboldsymbol{\xi}} =
{\nboldsymbol{\xi}/|\nboldsymbol{\xi}|}$ and
$\pi_{\nboldsymbol{\xi}^\perp}\mathbf{v} =
\mathbf{v}-(\mathbf{e}_{\nboldsymbol{\xi}}\cdot \mathbf{v})\mathbf{e}_{\nboldsymbol{\xi}}$ is the projection of $v$ onto the plane orthogonal to $\nboldsymbol{\xi}$.   
Figure \ref{fig1} shows a geometric interpretation of (\ref{algebraicop}).

\begin{figure}[ht]
\begin{center}
\begin{picture}(0,0)%
\includegraphics{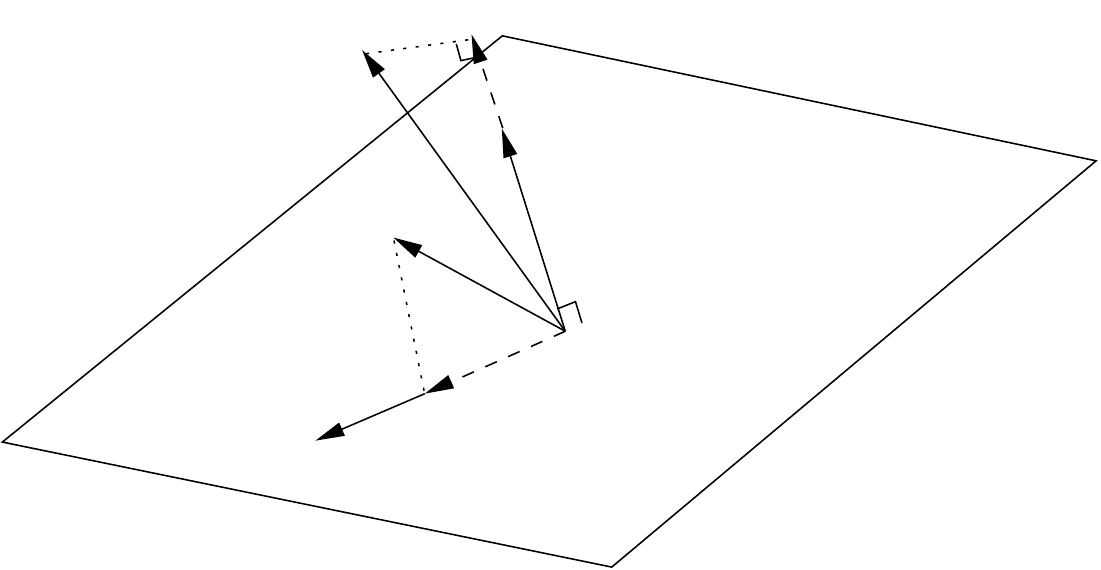}%
\end{picture}%
\setlength{\unitlength}{3947sp}%
\begingroup\makeatletter\ifx\SetFigFont\undefined%
\gdef\SetFigFont#1#2#3#4#5{%
  \reset@font\fontsize{#1}{#2pt}%
  \fontfamily{#3}\fontseries{#4}\fontshape{#5}%
  \selectfont}%
\fi\endgroup%
\begin{picture}(5274,2728)(3512,-3530)
\put(6076,-1636){\makebox(0,0)[lb]{\smash{{\SetFigFont{12}{14.4}{\rmdefault}{\mddefault}{\updefault}{\color[rgb]{0,0,0}$\nboldsymbol{\xi}$}%
}}}}
\put(5851,-2761){\makebox(0,0)[lb]{\smash{{\SetFigFont{12}{14.4}{\rmdefault}{\mddefault}{\updefault}{\color[rgb]{0,0,0}$\pi_{\nboldsymbol{\xi}^\perp}\hat{\mathbf{u}}_0(\nboldsymbol{\eta}_1)$}%
}}}}
\put(4801,-961){\makebox(0,0)[lb]{\smash{{\SetFigFont{12}{14.4}{\rmdefault}{\mddefault}{\updefault}{\color[rgb]{0,0,0}$\hat{\mathbf{u}}_0(\nboldsymbol{\eta}_2)$}%
}}}}
\put(5026,-3136){\makebox(0,0)[lb]{\smash{{\SetFigFont{12}{14.4}{\rmdefault}{\mddefault}{\updefault}{\color[rgb]{0,0,0}$i\hat{\mathbf{u}}_0(\nboldsymbol{\eta}_1)\odot_{\nboldsymbol{\xi}} \hat{\mathbf{u}}_0(\nboldsymbol{\eta}_2)$}%
}}}}
\put(5026,-1861){\makebox(0,0)[lb]{\smash{{\SetFigFont{12}{14.4}{\rmdefault}{\mddefault}{\updefault}{\color[rgb]{0,0,0}$\hat{\mathbf{u}}_0(\nboldsymbol{\eta}_1)$}%
}}}}
\put(5176,-3436){\makebox(0,0)[lb]{\smash{{\SetFigFont{12}{14.4}{\rmdefault}{\mddefault}{\updefault}{\color[rgb]{0,0,0}\colorbox{white}{$=(\mathbf{e}_{\nboldsymbol{\xi}} \cdot \hat{\mathbf{u}}_0(\nboldsymbol{\eta}_2))\pi_{\nboldsymbol{\xi}^\perp}\hat{\mathbf{u}}_0(\nboldsymbol{\eta}_1)$}}%
}}}}
\put(5938,-1287){\makebox(0,0)[lb]{\smash{{\SetFigFont{12}{14.4}{\rmdefault}{\mddefault}{\updefault}{\color[rgb]{0,0,0}\colorbox{white}{$\big(\mathbf{e}_{\nboldsymbol{\xi}}\cdot\hat{\mathbf{u}}_0(\nboldsymbol{\eta}_2)\big)\mathbf{e}_{\nboldsymbol{\xi}}$}}%
}}}}
\end{picture}%
\end{center}
\caption{Geometric interpretation of $\hat{\mathbf{u}}_0(\nboldsymbol{\eta}_1)\odot_{\nboldsymbol{\xi}} \hat{\mathbf{u}}_0(\nboldsymbol{\eta}_2)$.}\label{fig1}
\end{figure}

This stochastic cascade provides a weak solution to (\ref{NSeqn}) for initial data
$\mathbf{u}_0$ as an expected value of a cascade product under the algebraic operation
$\odot_{\nboldsymbol{\xi}}$, 
\begin{equation}
\label{expectation}
\hat{\mathbf{u}}(\nboldsymbol{\xi},t) = |\nboldsymbol{\xi}|^{-2}\mathbb{E}_{\nboldsymbol{\xi}}\mathbf{X}(\tau_t),
\end{equation}
where $\mathbf{X}(\tau_t)$ refers to a $\odot_{\nboldsymbol{\xi}}$-product of initial data and forcing at 
wave numbers determined by the branching Markov chain 
over nodes of the genealogical tree
$\tau_t$ at time $t$,
{\it provided that the indicated expectations exist.} The latter existence of
expected values is an essential proviso whether the cascade is thinned
or not. 
   
Of course the indicated expected values are to be interpreted
component-wise when applied to vector quantities.  In addition to the obvious
decay required on the magnitude of the algebraic multiplications
for existence of expectation integrals, the rate of growth of the tree
is also a significant issue.  Specifically, we will be interested in the possibility
of an {\it explosion event} in which infinitely many branchings occur
within finite time.

 The phenomena of \lq\lq explosion\rq\rq of Markov processes and
its relationship to uniqueness/non-uniqueness of solutions to the corresponding Kolmogorov equations is well-known in the theory
of stochastic processes; e.g., see \cites{TS1969, HFSW1968, RBEW2009}.  
The essence of explosion
is that the stochastic process may leave the state space in finite but
random time $\zeta$, and then be instantaneously returned to the state
space according to some arbitrary distribution.  Moreover this {\it
regenerative extension} 
can then be repeated to obtain a Markov process whose transition probabilities
also satisfy the {\it same} Kolmogorov (backward) equations. However, this is 
a linear Markov process theory that does not directly apply to (\ref{NSeqn}).
Nonetheless, the associated branching process is a Markov process for
which, in the absence of thinning, explosion cannot a priori be ruled out. 
In this context we consider the following.
 
 \begin{definition} 
  \label{expltime}
  The {\it explosion time} of the Fourier mode cascade genealogy {originating at $\nboldsymbol{\xi_0}$}
 is the (possibly infinite) random variable given by
 $$
 \zeta={ \zeta(\nboldsymbol{\xi_0}) =} \lim_{n\to\infty}\min_{|s|=n}\sum_{j=1}^n|\mathbf{W}_{s|j}|^{-2}T_{s|j},
$$
 where for each $n\ge 1$, $|s|=n$, denotes a genealogical sequence 
 $s = (s_1,...,s_n)\in\{1,2\}^n$, and for $j\le n$, 
 $s|j = (s_1,\dots,s_j)$ is the restriction of $s$ to the first $j$ 
 generations.  The random variables $T_{s},
 s\in\cup_{n=1}^\infty\{1,2\}^n$, are i.i.d.\ mean one exponentially
 distributed random variables independent of the Fourier modes
 ${\mathbf{W}_s}, s\in\cup_{n=1}^\infty\{1,2\}^n.$  The event $[\zeta
 < \infty]$ is referred to as an  {\it explosion event}. 
 \end{definition}
 
Define also $$\zeta_n = \min_{|s|=n}\sum_{j=1}^n|\mathbf{W}_{s|j}|^{-2}T_{s|j},$$
and note that by monotonicity one has $\zeta = \lim_{n\to\infty} \zeta_n.$

Various general conditions for explosion/non-explosion will be given
in the next section. A consequences of non-explosion is as follows.

While the full connection is incomplete, the following provides some
further evidence of a connection between explosion and the uniqueness
problem for (\ref{NSeqn}).

\begin{prop}\label{prop:uniq}
If there is no explosion then the stochastic cascade solution provides
the unique mild solution to Navier-Stokes equations whenever the indicated expectations exist.
\end{prop}
\begin{proof}
If explosion does not occur then the stochastic cascade is recursively
well-defined and the same martingale (inductive) arguments of 
\cite{YLAS1997} may be applied.
\end{proof}

\begin{remark}
For the converse case that explosion can be shown to occur, one may 
construct the regenerative extensions as mentioned above.  However 
it is not clear whether or not this is a pathway to non-uniqueness. 
\end{remark}

{
\begin{remark}\label{crit_cond}
 Note that if we re-scale NSE according to scaling (\ref{scaling}), then the explosion time random variable for the re-scaled cascade, $\zeta_r$, and $\zeta$ from Definition \ref{expltime} have the same distribution. Thus the non-explosion provides a {\em critical} (or scaling invariant) condition for uniqueness of the Navier-Stokes equations.
 \end{remark}
}

From the point of view of uniqueness of solutions to 
associated PDEs, it is interesting to note that the absence of explosion does 
correspond to the uniqueness of solutions for an evolution
equation.  In order to state this equation we first define an operator
$\Lambda$ by 
\begin{equation}
\label{calderon}
\Lambda (f) (\mathbf{x}) = {\mathcal{F}}^{-1} {(|\nboldsymbol{\xi}| \hat{f}(\nboldsymbol{\xi}))} (\mathbf{x}),
\end{equation}
where ${\mathcal{F}}^{-1}$ denotes inverse Fourier transform.

This operator acts to increase the higher frequency oscillations of
$f(\mathbf{x})$ by the same magnitude as differentiation.  As this operator is
closely related to differentiation, it is known as a
pseudo-differential operator.  The evolution equation associated to the
branching process as follows, which contains
this pseudo-differential operator, is as follows. 

%%%%%
%%%%%
%%%%%  NEW VERSION OF PROPOSITION
%%%%%
%%%%%

%{This is a new version of the proposition.  In particular it uses %a slightly different
%notation from before.}

\begin{prop}
\label{prop:PsiDE}
Let $h=h_d$ or $h=h_b$ and assume that the pseudo-differential equation
$$
\left\{\begin{gathered}
  \frac{\partial v}{\partial t} = \Delta v + c \Lambda (v^2)\\
   v(\mathbf{x},0) = {\mathcal{F}}^{-1}(h) (\mathbf{x})
\end{gathered}\right.
$$
with $c=\pi^3$ or $c=2\pi$ respectively, has a unique mild solution satisfying
$|\hat{v}(\nboldsymbol{\xi},t)| \leq h(\nboldsymbol{\xi})$ for all $t\geq 0.$  Then explosion does not occur, i.e.
$
\mathbb{P}([\zeta < \infty]) = 0.
$
\end{prop}

\begin{proof}
Fix $h=h_d,$ or $h=h_b$ and  $c=\pi^3$ or $c=2\pi$ respectively.
Let $Z(\nboldsymbol{\xi},t)$ denote the number of offspring by time $t\ge 0$
for the Fourier mode cascade genealogy starting at time $t=0$ at frequency 
$\mathbf{0}\neq \nboldsymbol{\xi}\in \mathbb{R}^k$.   
In particular,
after an exponentially distributed time  with
parameter $|\nboldsymbol{\xi}|^2$,
 the parent particle is replaced by two particles of  
 frequencies $\nboldsymbol{\eta}, \nboldsymbol{\xi}-\nboldsymbol{\eta}$ where $\nboldsymbol{\eta}$ has
probability density 
\begin{equation}
\label{def:H}
H(\nboldsymbol{\eta}|\nboldsymbol{\xi})=h(\nboldsymbol{\xi}-\nboldsymbol{\eta})/(c h*h(\nboldsymbol{\xi})).
\end{equation}

For $k\geq 1,$ let 
\begin{equation}
{m}(\nboldsymbol{\xi},t,k) = P_{\nboldsymbol{\xi}}(Z(\nboldsymbol{\xi},t)=k).
\end{equation}
In particular ${m}(\nboldsymbol{\xi},t,1) = \exp(-|\nboldsymbol{\xi}|^2 t).$
Explosion does not occur if and only if
$P_{\nboldsymbol{\xi}}(Z(\nboldsymbol{\xi},t) = \infty)=0,$
or equivalently
$
\sum_{k=0}^\infty {m}(\nboldsymbol{\xi},t,k) = 1.$
%which we now show to be the case.

Let $v(\nboldsymbol{\xi},t,k) = h(\nboldsymbol{\xi}) {m}(\nboldsymbol{\xi},t,k).$ 
Then trivially $v(\nboldsymbol{\xi},t,1) = h(\nboldsymbol{\xi})\exp(-|\nboldsymbol{\xi}|^2 t).$  Moreover,
 conditioning on the time of the first
branching, it follows using \eqref{hstarh} 
that for $k\geq 2, v(\nboldsymbol{\xi},t,k)$ satisfies the integral equation
\begin{equation}
\label{branching}
v(\nboldsymbol{\xi},t,k) = c |\nboldsymbol{\xi}| \sum_{j=1}^{k-1} \int_0^t \int_{\R^3} v(\nboldsymbol{\eta},t-s,j) v(\nboldsymbol{\xi}-\nboldsymbol{\eta},t-s,k-j) d\nboldsymbol{\eta} e^{-|\nboldsymbol{\xi}|^2 s} ds,\quad v(\nboldsymbol{\xi},0,k) = 0.
\end{equation}

Let $\hat{v}(\nboldsymbol{\xi},t) = \sum_{k=1}^\infty v(\nboldsymbol{\xi},t,k).$  This series converges
 since all terms are 
non-negative and the partial sums are clearly bounded above by $h(\nboldsymbol{\xi})$.  
Note also that 
\begin{equation}
\label{eq:uniqueness}
\hat{v} (\nboldsymbol{\xi},t) = h(\nboldsymbol{\xi})  \sum_{k=0}^\infty {m}(\nboldsymbol{\xi},t,k) .
\end{equation}
Summing (\ref{branching}) for $k\geq 2$, 
and adding the missing term $v(\nboldsymbol{\xi},t,1)$ one finds that $\hat{v}$ satisfies
\begin{equation}
\label{totalfourier}
\hat{v}(\nboldsymbol{\xi},t) =h(\nboldsymbol{\xi}) \exp(-|\nboldsymbol{\xi}|^2 t) + c|\nboldsymbol{\xi}| \int_0^t \int_{\R^3} 
\hat{v}(\nboldsymbol{\eta},t-s) \hat{v}(\nboldsymbol{\xi}-\nboldsymbol{\eta},t-s) d\nboldsymbol{\eta} e^{-|\nboldsymbol{\xi}|^2 s} ds
\end{equation}
It follows that  $v(\mathbf{x},t),$ the inverse Fourier transform of $\hat{v}(\nboldsymbol{\xi},t),$ satisfies
the following reaction-diffusion equation
 of \cite{SM2001},
\begin{equation}
\label{total}
\frac{\partial v}{\partial t}(\mathbf{x},t) = \Delta v(\mathbf{x},t) +
c \Lambda(v^2)(\mathbf{x},t)
\end{equation}
with initial data $v(\mathbf{x},0) = {\mathcal{F}}^{-1}(h)(\mathbf{x}).$  

One may easily check that with $\check{h} = {\mathcal{F}}^{-1}(h),$
\begin{equation}
\Delta {\check{h}} + c \Lambda ({\check{h}})^2(\mathbf{x}) = 0.
\end{equation}
By hypothesis, ${v}(\mathbf{x},t) = {\check{h}}(\mathbf{x})$ %is the unique positive solution 
%(\ref{total}) bounded by
%$h(\xi)$ to
%(\ref{total}), then 
and thus, using (\ref{eq:uniqueness}),
\begin{equation}
\label{welldefined}
P_{\nboldsymbol{\xi}}(\zeta > t) = P_{\nboldsymbol{\xi}}(Z(\nboldsymbol{\xi},t) < \infty) = \sum_{k=1}^\infty {m}(\nboldsymbol{\xi},t,k) = 1,\quad \forall \nboldsymbol{\xi},t.
\end{equation}
\end{proof}

\begin{remark}
Note that since $h_d$ and $h_b$ are radially symmetric, the cascade, hence $m(\nboldsymbol{\xi},t,k)$ and
$m(\nboldsymbol{\xi},t) \equiv \sum_{k=0}^\infty m(\nboldsymbol{\xi},t,k)$ are also radially 
symmetric.  In particular, it follows that with $H$ defined by (\ref{def:H}),
\begin{equation}
\label{mequationns}
m{}(|\nboldsymbol{\xi|},t) = \exp(-|\nboldsymbol{\xi}|^2 t) + |\nboldsymbol{\xi}|^2 \int_0^t \int_{\R^3} 
{m}(|\nboldsymbol{\eta}|,t-s) {m}(|\nboldsymbol{\xi}-\nboldsymbol{\eta}|,t-s) 
H(\nboldsymbol{\eta}|\nboldsymbol{\xi} )\,\rmd\nboldsymbol{\eta} {\rm{e}}^{-|\nboldsymbol{\xi}|^2 s} \,\rmd
s
\end{equation}
\end{remark}

%%%%
%%%% End of changes on Sect 2 by ET
%%%%

\begin{remark}
In 2007 Chris Orum has announced the equation of Proposition \ref{prop:PsiDE},
 and its role in the explosion problem in a session of the 32nd Conference on Stochastic Processes and their Applications, Champaign-Urbana.  However the uniqueness/explosion problems remain unsolved for general majorizing kernels $h$ as
initial data.  
\end{remark}

We conclude this section with a small further elaboration on the 
probabilistic significance of the two kernels $h_d,h_b$.  Additional
features are discussed in the appendix. 
 
Let
\begin{equation} 
\label{pseudodiff} 
a(\nboldsymbol{\xi}) = \frac{h(\nboldsymbol{\xi}) |\nboldsymbol{\xi}|^2}{h*h(\nboldsymbol{\xi})}.
\end{equation}
Then, for either  kernel $ h = h_d$ or  $h = h_b$, 
one has that
 $a_i(\nboldsymbol{\xi}) = c^{-1}|\nboldsymbol{\xi}|$, $i=1,2$
  (\ref{pseudodiff}) with $c=\pi^3$ or $c=2\pi$ respectively, defines the same pseudo-differential operator given by a positive multiple 
  of  $\sqrt{-\Delta}$. However 
the following two propositions dramatically distinguish the associated branching
Markov chains. 

\begin{prop} 
\label{dilogbrw}
Assume that $h(\nboldsymbol{\xi})\equiv h_d(\nboldsymbol{\xi}) = |\nboldsymbol{\xi}|^{-2}, \mathbf{0}\neq \nboldsymbol{\xi}\in\mathbb{R}^3.$ 
Then for
each $s\in\{1,2\}^\infty$, the sequence $\{{|\mathbf{W}_{s|j+1}|/|\mathbf{W}_{s|j}|} : j = 0,1,\dots\}$, $\mathbf{W}_{s|0} = \nboldsymbol{\xi}$, is an i.i.d.\ sequence under $P_{\nboldsymbol{\xi}}$, such that,
\begin{eqnarray*}
&& P_{\nboldsymbol{\xi}}\Big(\frac{|\mathbf{W}_{s|j+1}|}{ |\mathbf{W}_{s|j}|}\in dr\Big) = 2\pi^{-2}\ln \Big|\frac{1+r}{ 1-r}\Big|\,\frac{dr}{ r}, \quad r > 0.\\
 &&  P_{\nboldsymbol{\xi}}\Big(\ln \frac{|\mathbf{W}_{s|j+1}|}{  |\mathbf{W}_{s|j}|}\in dt\Big) = 2\pi^{-2}\ln|\coth({t/2})|\,dt, \quad t\in \R.
 \end{eqnarray*}
\end{prop}
%\vskip .25in
%\noindent
\begin{proof}
Part (ii) is an immediate consequence of (i) by a change of variables formula.
As noted above, the  distribution of $|\mathbf{W}|$ was computed as  (1.22) of \cite{YLAS1997}.  Essentially
the same calculations apply to the ratios of magnitudes as follows:
For non-negative and integrable $g$ on $(0,\infty)$,
one has, since $h*h(\nboldsymbol{\xi}) = \pi^3 /|\nboldsymbol{\xi}|$, and  using (\ref{brwtypes}) that
\begin{eqnarray*}
\mathbb{E}_{\nboldsymbol{\xi}} g\Big(\frac{|\mathbf{W}_{s|1}|}{ |\nboldsymbol{\xi}|}\Big) &=&
\pi^{-3}|\nboldsymbol{\xi}|\int_{\mathbb{R}^3} g\Big(\frac{|\nboldsymbol{\eta}|}{|\nboldsymbol{\xi}|}\Big)\frac{d\nboldsymbol{\eta}}{|\nboldsymbol{\eta}|^2|\nboldsymbol{\xi}-\nboldsymbol{\eta}|^2}\\
&=& \pi^{-3}\int_{\mathbb{R}^3}g(|\mathbf{v}|)\frac{d\mathbf{v}}{|\mathbf{v}|^2|\mathbf{u}-\mathbf{v}|^2}, \quad u =\frac {\nboldsymbol{\xi}}{|\nboldsymbol{\xi}|}\\
&=& \pi^{-3}\int_{\mathbb{R}^3}g(|\mathbf{v}|)\frac{d\mathbf{v}}{|\mathbf{v}|^2(|\mathbf{v}|^2 -2\,\mathbf{u}\cdot\mathbf{v} + 1)}\\
&=&\pi^{-3}\int_0^\infty \int_{|\mathbf{w}|=1}g(r)\frac{d\mathbf{w}\, dr}{ r^2-2r\,\mathbf{u}\cdot\mathbf{w} + 1}\\
&=&2\pi^{-2}\int_0^\infty \int_0^{\pi}g(r)\frac{\sin\phi \,d\phi\, dr}{ r^2-2r\cos\phi + 1}\\
&=& 2\pi^{-2}\int_0^\infty g(r)\ln\frac{|1+r|}{ |1-r|}\frac{dr}{ r}.
\end{eqnarray*}
\end{proof}

\

\begin{definition}
The multiplicative random walk  $\{|\mathbf{W}_{s|j+1}| : j = 0,1,\dots\}$ 
on $(0,\infty)$ will be referred to as the {\it dilogarithmic  random walk},
or {\it dilog} random walk for short.
\end{definition}

\

\begin{prop}
\label{besselbmc}
For arbitrary $\nboldsymbol{\xi}\in \mathbb{R}^3$, let $\mathbf{W}$ denote the random vector in $\mathbb{R}^3$ with
density
$$
H_b(\nboldsymbol{\eta} \,|\, \nboldsymbol{\xi}) = \frac{e^{|\nboldsymbol{\xi}|}}{2\pi} \frac{e^{-|\nboldsymbol{\eta}|} e^{-|\nboldsymbol{\xi}-\nboldsymbol{\eta}|}}{|\nboldsymbol{\eta}||\nboldsymbol{\xi}-\nboldsymbol{\eta}|}.
$$
Then
$$
\mathbb{P}({|\mathbf{W}|} \in {\rm{d}}r)
= 
\left\{
\begin{array}{ll}
\frac{1}{|\nboldsymbol{\xi}|}e^{-2r} (e^{2|\nboldsymbol{\xi}|} -1) & {\rm{for}}~~r\geq |\nboldsymbol{\xi}| \\[.1in]
\frac{1}{|\nboldsymbol{\xi}|}(1-e^{-2r}) & {\rm{for}}~~~0\leq r \leq |\nboldsymbol{\xi}|.
\end{array}
\right.
$$
In particular,  along any path $s$, the sequence $|\mathbf{W}_\emptyset| = |\nboldsymbol{\xi}|, |\mathbf{W}_{s|1}|, |\mathbf{W}_{s|2}|,\dots,$
is a Markov chain with stationary  transition probability density 
$$p(u,v) = 
\left\{
\begin{array}{ll}
\frac{1}{u}e^{-2v} (e^{2u} -1) & {\rm{for}}~~v\geq u, u > 0 \\[.1in]
\frac{1}{u}(1-e^{-2v}) & {\rm{for}}~~~0\leq v \leq u, u > 0.
\end{array}
\right.
$$
Moreover,
\begin{equation}
\label{meanreversion}
{{\mathbb{E}_{\nboldsymbol{\xi}}|\mathbf{W}_1| = \frac{|\nboldsymbol{\xi}|+1}{2}.}}
\end{equation}
\end{prop}
\begin{proof}
For arbitrary $\nboldsymbol{\xi}\in \mathbb{R}^3\setminus\{\mathbf{0}\}$, let $\mathbf{W}$ denote the random vector in $\mathbb{R}^3$ with density
$$
H_b(\nboldsymbol{\eta}\,|\,\nboldsymbol{\xi}) = \frac{e^{|\nboldsymbol{\xi}|}}{2\pi} \frac{e^{-|\nboldsymbol{\eta}|} e^{-|\nboldsymbol{\xi}-\nboldsymbol{\eta}|}}{|\nboldsymbol{\eta}||\nboldsymbol{\xi}-\nboldsymbol{\eta}|}.
$$
Let $u=|\nboldsymbol{\xi}|,$ and use
spherical coordinates  with
$\rho = |\nboldsymbol{\eta}|$.  Direct calculations exploiting the cylindrical symmetry of the distribution give
\begin{align*}
\mathbb{P}(|\mathbf{W}|> r) =& \frac{e^{u}}{2\pi} (2\pi) \int_{r}^\infty e^{-\rho} \rho
\int_0^\pi 
\frac{\exp(-(\rho^2 - 2\rho u\cos\phi + u^2)^{1/2})} {(\rho^2 - 2\rho u\cos\phi + u^2)^{1/2}}
{\sin \phi}\, d\phi d\rho, \\
=& \frac{e^{u}}{u} (-1) 
 \int_{r}^\infty e^{-\rho} (e^{-(\rho+u)} - e^{-|\rho-u|}) \,d\rho, \\
 =&
\frac{e^{u}}{u} 
\left[
\int_{r}^\infty e^{-\rho} e^{-|\rho-u|} \,d\rho
- 
\frac{e^{-u}}{2} e^{-2r} 
\right]~.
\end{align*}
For $r\geq u,$ the integral in the last expression is 
$
({e^{u}}/{2}) e^{-2r}
$ so that in this case, 
$$
\mathbb{P}(|\mathbf{W}| > r) = \frac{e^{-2r}}{2 u} (e^{2 u} - 1)~.
$$
Similarly, for  $r\leq u$ one obtains that
$$
\left[\int_r^u + \int_u^\infty\right]  e^{-\rho} e^{-|\rho-u|} \,d\rho
= 
e^{-u}(u-r + \frac{1}{2}),
$$
so that in this case,
$$
\mathbb{P}(|\mathbf{W}|> r) 
=
\frac{1}{2 u}(1-e^{-2r}) + \frac{u-r}{u} ~,
$$
and the result follow by differentiation.  
\end{proof}

\begin{remark}
The calculation of the marginal distribution of $|\mathbf{W}_1|$ can be found
 in (\cite{YLAS1997}, Proposition 2.1) for $h=h_d$.  A similar calculation was 
 provided here for ease of reference.  
%The section on self-similar explosion
%contains various properties of the dilogarithmic branching random %walk.
%on $(0,\infty)$ defined by independent ratios (i.e., multiplicative %increments)
%for this kernel.
 The Markov property follows by exploiting the construction together with the 
form of the transition probabilities as functions of the norms; see (\cite{RBEW2009}, 
pp 502-503).  
\end{remark}

\begin{remark}\label{rem:dilog}
 Euler's {\it dilogarithmic function} may be defined by
\begin{equation}
\label{dilog}
\Li_2(r) = -\int_0^r\ln(1-u)\frac{du}{ u}, \quad  r < 1.
\end{equation} 
The dilogarithmic function is a special case of polylogarthmic functions
$\Li_s(x)$ whose domain of definition may be extended to include complex values
of both $x$ and $s$. An extensive literature is available for properties and relationships 
between polylogarithmic functions, with connections to Bose-Einstein and Fermi-Dirac
statistics, e.g., see \cites{ML1995, ML1997, AK1995}  
\end{remark} 

The explosion problem is solved
for the Bessel kernel in the appendix {(Theorem \ref{th:nonexplosionBessel})}, but it remains quite illusive
for the dilogarthmic kernel.  However, as
 will be seen in the next
section, the dilogarithmic kernel is somewhat singled out by 
the self-similarity cascade.  
 
\section{Self-Similar (Navier-Stokes) Cascade \& And Its Associated Explosion Problem}\label{sec3}

In this section, we obtain a stochastic cascade associated to the Navier-Stokes equation 
when self similar solutions are considered.  It should be remarked from the outset
that the kernel $H_d$ occurs naturally in this situation, as a direct consequence
of the scaling properties of the solutions of the Navier-Stokes equations.  We present first
the mild formulation of the Fourier transform of the Navier-Stokes equations for self similar
solutions.  A probabilistic representation for the solution of the resulting equation is given
in terms of what we call the {\textit{self similar cascade}}.  We show that important statistical 
properties of this self similar cascade and the Navier-Stokes cascade obtained using the
dilogarthmic kernel $H_d$ are identical.

As noted in the introduction,
the scaling invariance of the Navier-Stokes equations show that if  $\mathbf{u}(\mathbf{x},t), p(\mathbf{x},t)$  is  a solution 
then for any $r > 0, \mathbf{u}_r(\mathbf{x},t) \equiv r \mathbf{u}(r \mathbf{x}, r^2 t),$
$p_r(\mathbf{x},t) = r^2 p(r \mathbf{x}, r^2 t)$
is also a solution of the Navier-Stokes equations. 
Assuming the initial data is also scale invariant, uniqueness would imply the self-similarity property $\mathbf{u}_r = \mathbf{u}.$
In the Fourier domain, this scale invariance 
corresponds to, with $\mathbf{v} = \mathbf{u}_r,$
$$
\hat{\mathbf{v}}(\nboldsymbol{\xi},t) = \frac{1}{r^2} \hat{\mathbf{u}}(\frac{\nboldsymbol{\xi}}{r}, r^2 t),
$$
so, with $r=|\nboldsymbol{\xi}|,$
\begin{equation}\label{self-sim_scaling}
\hat{\mathbf{v}}(\nboldsymbol{\xi},t) = \frac{1}{|\nboldsymbol{\xi}|^2} \hat{\mathbf{u}}(\mathbf{e}_{\nboldsymbol{\xi}}, |\nboldsymbol{\xi}|^2 t)
\end{equation}
where $\mathbf{e}_{\nboldsymbol{\xi}}=\nboldsymbol{\xi}/|\nboldsymbol{\xi}|.$ 
%Recall that with the adopted  definition for the Fourier transform 
%$\hat{f}(\nboldsymbol{\xi}) = (2\pi)^{-3/2} \int \exp(-i \mathbf{x}\cdot\nboldsymbol{\xi}) f(\mathbf{x}) \,d \mathbf{x},$  
%one has
%$\widehat{(fg)} (\nboldsymbol{\xi}) = (2\pi)^{-3/2} \hat{f}*\hat{g}(\nboldsymbol{\xi}).$  
Thus, since $\hat{\mathbf{v}}$  {satisfies \eqref{mildNSE} it} follows that a self similar solution of
the Navier-Stokes equations satisfies
\begin{align*}
\hat{\mathbf{u}}(\mathbf{e}_{\nboldsymbol{\xi}}, |\nboldsymbol{\xi}|^2t)) &= 
e^{-t|\nboldsymbol{\xi}|^2} \hat{\mathbf{u}}_0(\mathbf{e}_{\nboldsymbol{\xi}}) \\
&+
{(2\pi)^{-3/2}}  \int_0^t e^{-|\nboldsymbol{\xi}|^2(t-s)} |\nboldsymbol{\xi}|^3 \int
\hat{\mathbf{u}}(\mathbf{e}_{\nboldsymbol{\eta}}, |\nboldsymbol{\eta}|^2s) \odot_{\nboldsymbol{\xi}} \hat{\mathbf{u}}(\mathbf{e}_{\nboldsymbol{\xi}-\nboldsymbol{\eta}},|{\nboldsymbol{\xi}} - \nboldsymbol{\eta}|^2 s) \frac{1}{|\nboldsymbol{\eta}|^2|{\nboldsymbol{\xi}}-\nboldsymbol{\eta}|^2}
\,d\nboldsymbol{\eta} d s.
\end{align*}

%Recall that the operation denoted by $\odot_{\nboldsymbol{\xi}}$ involves the projection of $\hat{v}$ on the direction of
%$\mathbf{e}_{\nboldsymbol{\xi}}$ and on the plane perpendicular to this direction.  Notice that
%there is no dependence on the magnitude of $\nboldsymbol{\xi}$ in this operation. (Actually the factor $|\nboldsymbol{\xi}|$
%is introduced to account for the divergence term in the Fourier space.) 

The change of variables $\nboldsymbol{\eta}=|\nboldsymbol{\xi}|\nboldsymbol{\eta}', s'=|\nboldsymbol{\xi}|^2 s,$ and with $\lambda=|\nboldsymbol{\xi}|^2 t,$ gives
\begin{align*}
\hat{\mathbf{u}}(\mathbf{e}_{\nboldsymbol{\xi}}, \lambda) &= 
e^{-\lambda} \hat{\mathbf{u}}_0(\mathbf{e}_{\nboldsymbol{\xi}}) \\
&+
{(2\pi)^{-3/2}}  \int_0^\lambda e^{-(\lambda-s)} \int
\hat{\mathbf{u}}(\mathbf{e}_{\nboldsymbol{\eta}}, |\nboldsymbol{\eta}|^2s) \odot_{\nboldsymbol{\xi}} \hat{\mathbf{u}}(\mathbf{e}_{\mathbf{e}_{\nboldsymbol{\xi}}-\nboldsymbol{\eta}},|\mathbf{e}_{\nboldsymbol{\xi}} - \nboldsymbol{\eta}|^2 s) \frac{1}{|\nboldsymbol{\eta}|^2|\mathbf{e}_{\nboldsymbol{\xi}}-\nboldsymbol{\eta}|^2}
\,d\nboldsymbol{\eta} d s.
\end{align*}
Recall that $H_{d}(\nboldsymbol{\eta}|\mathbf{e}_{\nboldsymbol{\xi}}) = (\pi^3 |\nboldsymbol{\eta}|^2 |\mathbf{e}_{\nboldsymbol{\xi}} - \nboldsymbol{\eta}|^2)^{-1}$ 
%and recall $\int H(\mathbf{e}_{\nboldsymbol{\xi}},\nboldsymbol{\eta}) d\nboldsymbol{\eta} = 1.$
%Then one has
so one has
\begin{equation}
\label{mildssns}
\hat{\mathbf{u}}(\mathbf{e}_{\nboldsymbol{\xi}}, \lambda) = 
e^{-\lambda} \hat{\mathbf{u}}_0(\mathbf{e}_{\nboldsymbol{\xi}}) +
{(\pi/2)^{3/2}}  \int_0^\lambda e^{-(\lambda-s)} \int
\hat{\mathbf{u}}(\mathbf{e}_{\nboldsymbol{\eta}}, |\nboldsymbol{\eta}|^2s) \odot_{\nboldsymbol{\xi}} \hat{\mathbf{u}}(\mathbf{e}_{\mathbf{e}_{\nboldsymbol{\xi}}-\nboldsymbol{\eta}},|\mathbf{e}_{\nboldsymbol{\xi}} - \nboldsymbol{\eta}|^2 s) H_{d}(\nboldsymbol{\eta}|\mathbf{e}_{\nboldsymbol{\xi}})% \frac{1}{|\nboldsymbol{\eta}|^2|e_\xi-\nboldsymbol{\eta}|^2}
d\nboldsymbol{\eta} d s.
\end{equation}
We refer to the parameter $\lambda > 0$ as the {\it similarity horizon}. 

A probabilistic interpretation for (\ref{mildssns}) follows similar steps as those introduced before.
Consider a binary tree rooted at $\emptyset$ with vertices indexed by 
${\mathcal{V}} =\cup_{n\geq 1}\{1,2\}^n$ -- see Figure \ref{fig2} for an illustration.  Denote by $\partial{\mathcal{V}} = \{1,2\}^{\mathbb{N}}.$
Elements in each of these sets are denoted by $s$ and $ {\langle s\rangle}$ respectively.
Let
$\{ T_s, s\in{\mathcal{V}}\}$ be a collection of i.i.d.\ random
variables with an exponential distribution with parameter 1. Given a direction 
$\mathbf{e}_{s}$, let
%$\nboldsymbol{\eta}_{s1}$
$\tilde{\mathbf{W}}_{s1}$  be a random variable with distribution
$H_{d}(\nboldsymbol{\eta}|\mathbf{e}_{s}),$ set 
%$\nboldsymbol{\eta}_{s2}=\mathbf{e}_{s} - \nboldsymbol{\eta}_{s1}$ 
$\tilde{\mathbf{W}}_{s2} = \mathbf{e}_{s} -\tilde{ \mathbf{W}}_{s1}$
and for $j=1,2,$ define the directions 
$\mathbf{e}_{sj}=\tilde{\mathbf{W}}_{sj}/|\tilde{\mathbf{W}}_{sj}|.$ 
Finally, given a {\textit{horizon}} $\lambda_{s},$
define for $j=1,2$ $\lambda_{sj} = |\tilde{\mathbf{W}}_{sj}|^2( \lambda_{s} - T_{s}).$
On each  ${\langle s\rangle} \in \partial{\mathcal{V}},$ the branching process stops at level 
\begin{equation}
\label{finiteN}
N_{\langle s \rangle} = \inf\{m\geq0: \lambda_{\langle s|m\rangle} < T_{\langle s|m\rangle}\}.
\end{equation}

Completely analogous to \eqref{expectation}, the solution of \eqref{mildssns} is then given as an expected value of a recursive product involving the algebraic operation 
$\odot_{\mathbf{e}_{\nboldsymbol{\xi}}}$
provided this expectation is finite.  Furthermore, the evaluation of this recursive product can be done if and only if along any path in the binary tree, the random variable $N_{\langle s \rangle}$ defined in \eqref{finiteN}
is finite.

From the definition of the random variables one has %on the event $[N = n]$,
\begin{eqnarray*}
\lambda_{\langle s|n\rangle} - T_{\langle s|n\rangle}
&=&
((...(((\lambda_\emptyset - T_\emptyset)|\tilde{\mathbf{W}}_{\langle s|1 \rangle}|^2 - T_{\langle s|1 \rangle}) 
|\tilde{\mathbf{W}}_{\langle s|2 \rangle}|^2 - T_{\langle s|2 \rangle})...) 
|\tilde{\mathbf{W}}_{\langle s|n \rangle}|^2 - T_{\langle s|n \rangle}) \\
&=&
 (\prod_{k=0}^{n} |\tilde{\mathbf{W}}_{\langle s|k \rangle}|^2) \left( \lambda_\emptyset
 %\nboldsymbol{\eta}_{\langle s|0 \rangle}|^2} 
 -
\sum_{j=0}^n T_{{\langle s|j \rangle}} \frac{1}{\prod_{k=0}^{j} |\tilde{\mathbf{W}}_{\langle s|k \rangle}|^2}\right).
\end{eqnarray*}
where we have used that $|\tilde{\mathbf{W}}_{\langle s|0\rangle}|^2 =1.$
%At the root of the tree one needs to prescribe a vector $\nboldsymbol{\eta}_\emptyset$ from which an initial
%direction $e_\emptyset$ and magnitude $|\nboldsymbol{\eta}_\emptyset|$ are determined.  One also needs to
%prescribe a {\it similarity horizon} $\lambda_\emptyset.$  In a typical egg and chicken problem, one can prescribe
%$t$ and define $\lambda_\emptyset = |\nboldsymbol{\eta}_\emptyset|^2 t$ or viceversa.% Let's adopt for now that $t$
%and $\xi_0=\eta_\emptyset= \eta_{\langle v|0 \rangle}$ are given.  
Thus, for given $\lambda_\emptyset$ and $\langle s \rangle \in \mathcal{V},$ the event 
$[N_{\langle s \rangle} =n]$ 
equals the event
$$
\inf\{m\geq 0: \sum_{j=0}^m T_{{\langle s|j \rangle}} \frac{1}{\prod_{k=0}^{j} |\tilde{\mathbf{W}}_{\langle s|k \rangle}|^2} \geq \lambda_\emptyset \} = n.
$$

%Define the random variable, dependent only on the direction $\mathbf{e}_0=\mathbf{e}_{\mathbf{W}_{\langle s|0\rangle}}$
%\begin{equation}
%\label{ssexplosion}
%\zetas_n(\mathbf{e}_0) = \inf_{| s |=n}
%\sum_{j=0}^n T_{{\langle s|j\rangle}} \frac{1}{\prod_{k=0}^{j} |\Tilde{\mathbf{W}}_{\langle s|j\rangle}|^2}.
%\stackrel{\mathcal{D}}{=}\inf_{| s |=n}
%\sum_{j=0}^n T_{{\langle s|j\rangle}} \frac{1}{\prod_{k=0}^{j} |R_{j}|^2}
%\end{equation}
%Note that for fixed horizon $\lambda_0,$ one has the equality (of events)

This motivates the following definition.

\begin{definition}
\label{selfsimilarexplosion}
For a fixed unit vector $\mathbf{e}_0, $ the {\textit{similarity explosion horizon}} is the
(possibly infinite) random variable 
$$\zetas(\mathbf{e}_0) = 
\lim_{n\rightarrow \infty} \inf_{| s |=n}
\sum_{j=0}^n T_{{\langle s|j\rangle}} \frac{1}{\prod_{k=0}^{j} |\Tilde{\mathbf{W}}_{\langle s|j\rangle}|^2}.
$$
%\label{ssexplosion}
The {\textit{self similar explosion event}} is defined as $A_{\mathbf{e}_0}=
\cup_{m\geq1}[\zetas(\mathbf{e}_0)<m]$ 
so that ${\mathbb{P}}( A_{\mathbf{e}_0})$
is the probability of {\textit{self similar explosion}}.
\end{definition}

Note that with
$$
\zetas_n(\mathbf{e}_0) = \inf_{| s |=n}
\sum_{j=0}^n T_{{\langle s|j\rangle}} \frac{1}{\prod_{k=0}^{j} |\Tilde{\mathbf{W}}_{\langle s|j\rangle}|^2}
$$
one has, by monotone convergence, that
$\zetas(\mathbf{e}_0) =
\lim_{n\rightarrow \infty}
\zetas_n(\mathbf{e}_0).
$

%{This paragraph needs to be reworked - It might be better to make this a proposition 
%stating that the event of eventual explosion is independent of the direction and that the equality
%in distribution stated in \eqref{ssexplosion} is valid}
While the self-similar cascade construction is quite distinct from that of
the Navier-Stokes cascade, one may note that for fixed 
${\langle s\rangle} \in \partial{\mathcal{V}},$
the random variables $\tilde{R}_j = |\tilde{\mathbf{W}}_{\langle s|j \rangle}|, j\geq 1$ 
are i.i.d.\ with the dilogarithmic 
distribution with density
%{Only capitalize proper names convention.}
$$\mathcal{D}(r)
=
\frac{2}{\pi^2}
\frac{1}{r} \ln\left(\frac{|1+r|}{|1-r|}\right)
%= \int_0^\pi \tilde{H}(r,\theta) \rmd \theta.
$$
Indeed, since the distribution of 
$\tilde{\mathbf{W}}_{\langle s|1 \rangle}$ depends only in the unit vector $\mathbf{e}_0,$
the proof of Proposition~\ref{dilogbrw} shows that $\tilde{R}_1$ has the dilogarithmic distribution.  The claim
{for  $\tilde{R}_j$} follows by induction.

In order to relate the explosion problems for the self similar cascade and the 
Navier-Stokes cascade,
we have the following result.

\begin{prop}
\label{Prop:direction_independent}
For any $n\geq 0,$ the distribution of $\zetas_n$ is independent of the initial direction and
$$%\begin{equation}
%\label{ssexplosion}
\zetas_n(\mathbf{e}_0)% = %\inf_{| s |=n}
%\sum_{j=0}^n T_{{\langle s|j\rangle}} \frac{1}{\prod_{k=0}^{j} |\Tilde{\mathbf{W}}_{\langle s|j\rangle}|^2}.
\stackrel{\mathcal{D}}{=}\inf_{| s |=n}
\sum_{j=0}^n T_{{\langle s|j\rangle}} \frac{1}{\prod_{k=0}^{j} |\tilde{R}_{j}|^2}
$$%\end{equation}
where  $\tilde{R}_0=1,$ and $\{\tilde{R}_j\}_{j=1}^\infty$ is a sequence of i.i.d.\ random variables
with density $\mathcal{D}(r).$  
\end{prop}

\begin{proof}
Let $Q$ be an orthogonal 3 by 3 matrix and $\mathbf{e}$ a unit vector in $\R^3.$
Let $\tilde{\nboldsymbol{\eta}}, \nboldsymbol{\eta}^\sharp$ be random vectors
distributed according to $H_d(\nboldsymbol{\eta}|Q \mathbf{e})$ and 
$H_d(\nboldsymbol{\eta}|\mathbf{e})$ respectively.  It follows easily that
in distribution, $\tilde{\nboldsymbol{\eta}}$ and $\nboldsymbol{\eta}^\sharp$ are equal and thus independent of the particular
initial direction $\mathbf{e}$ used in $H.$  
The proof is completed, since as noted above, $\mathcal{D}(r)$ is the
density of $\tilde{R}_j$.
\end{proof}

As a consequence of Proposition \ref{Prop:direction_independent}, 
the distribution of the sequence $\lambda_{\langle v|j \rangle},
j\geq 1$ is  also independent of the initial direction $\mathbf{e}_0.$  

%Having noted the iid properties of the random variables $R_j,$ it should be
%clear that the explosion problem for the self similar process is the same as the explosion
%problem for the dilogarithmic random walk.  However the meaning
%of ``explosion'' is quite different for the two stochastic cascades.

\begin{figure}[ht]

\begin{center}
\begin{picture}(0,0)%
\includegraphics{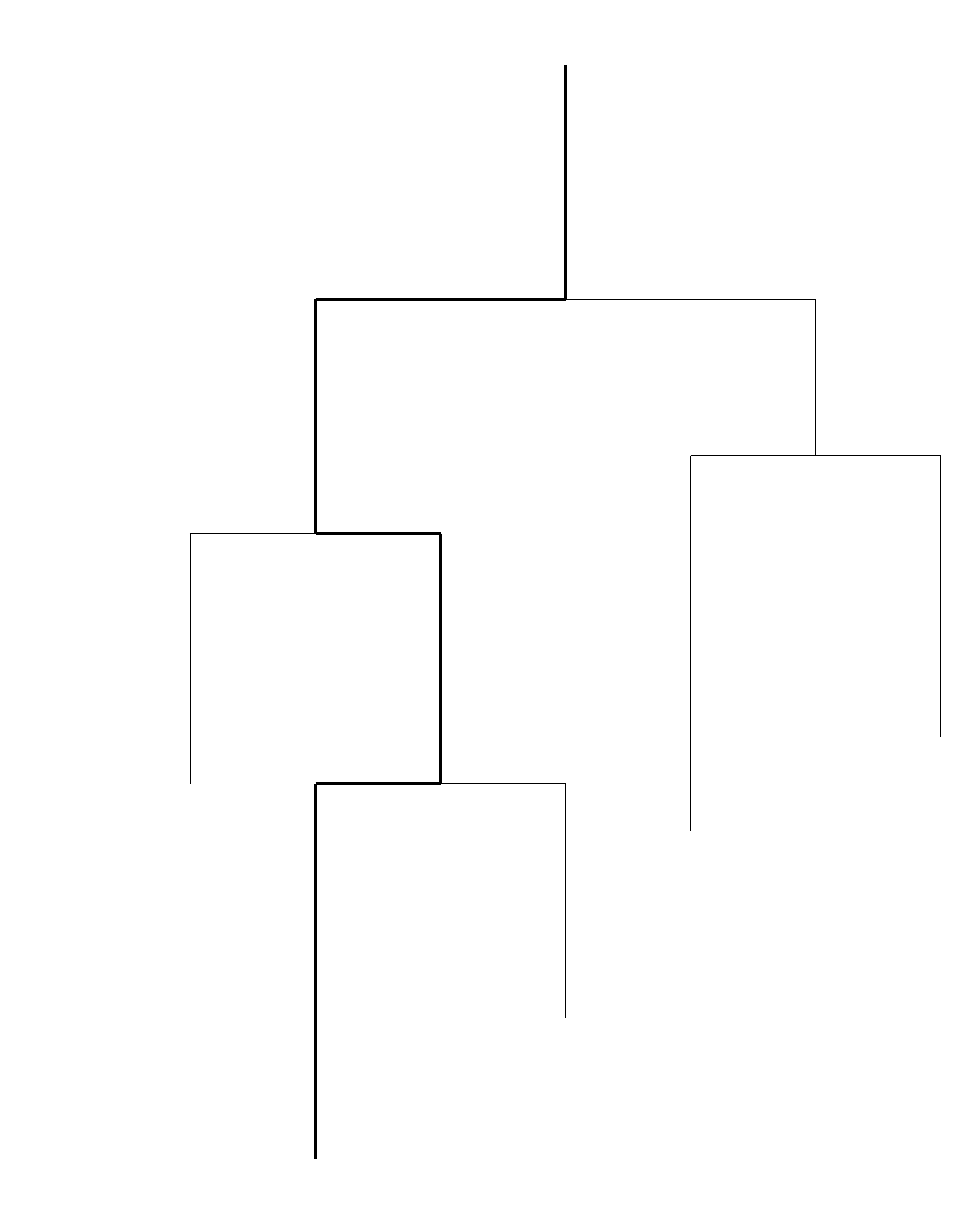}%
\end{picture}%
\setlength{\unitlength}{3947sp}%
\begingroup\makeatletter\ifx\SetFigFont\undefined%
\gdef\SetFigFont#1#2#3#4#5{%
  \reset@font\fontsize{#1}{#2pt}%
  \fontfamily{#3}\fontseries{#4}\fontshape{#5}%
  \selectfont}%
\fi\endgroup%
\begin{picture}(4680,5848)(3286,-4850)
\put(5851,839){\makebox(0,0)[lb]{\smash{{\SetFigFont{12}{14.4}{\rmdefault}{\mddefault}{\updefault}{\color[rgb]{0,0,0}$\mathbf{e}_{\nboldsymbol{\xi}}, \lambda$}%
}}}}
\put(6151,-136){\makebox(0,0)[lb]{\smash{{\SetFigFont{12}{14.4}{\rmdefault}{\mddefault}{\updefault}{\color[rgb]{0,0,0}$T_\emptyset$}%
}}}}
\put(4876,-886){\makebox(0,0)[lb]{\smash{{\SetFigFont{12}{14.4}{\rmdefault}{\mddefault}{\updefault}{\color[rgb]{0,0,0}$\lambda_1 > T_1$}%
}}}}
\put(7276,-886){\makebox(0,0)[lb]{\smash{{\SetFigFont{12}{14.4}{\rmdefault}{\mddefault}{\updefault}{\color[rgb]{0,0,0}$\lambda_2 > T_2$}%
}}}}
\put(5476,-2161){\makebox(0,0)[lb]{\smash{{\SetFigFont{12}{14.4}{\rmdefault}{\mddefault}{\updefault}{\color[rgb]{0,0,0}$\lambda_{12}>T_{12}$}%
}}}}
\put(7951,-2386){\makebox(0,0)[lb]{\smash{{\SetFigFont{12}{14.4}{\rmdefault}{\mddefault}{\updefault}{\color[rgb]{0,0,0}$\lambda_{22}<T_{22}$}%
}}}}
\put(3301,-2161){\makebox(0,0)[lb]{\smash{{\SetFigFont{12}{14.4}{\rmdefault}{\mddefault}{\updefault}{\color[rgb]{0,0,0}$\lambda_{11}<T_{11}$}%
}}}}
\put(3601,-3961){\makebox(0,0)[lb]{\smash{{\SetFigFont{12}{14.4}{\rmdefault}{\mddefault}{\updefault}{\color[rgb]{0,0,0}$\lambda_{121}>T_{121}$}%
}}}}
\put(6076,-3661){\makebox(0,0)[lb]{\smash{{\SetFigFont{12}{14.4}{\rmdefault}{\mddefault}{\updefault}{\color[rgb]{0,0,0}$\lambda_{122}<T_{122}$}%
}}}}
\put(6676,-2911){\makebox(0,0)[lb]{\smash{{\SetFigFont{12}{14.4}{\rmdefault}{\mddefault}{\updefault}{\color[rgb]{0,0,0}$\lambda_{21}<T_{21}$}%
}}}}
\put(4778,-4786){\makebox(0,0)[lb]{\smash{{\SetFigFont{12}{14.4}{\rmdefault}{\mddefault}{\updefault}{\color[rgb]{0,0,0}$\vdots$}%
}}}}
\end{picture}%
\caption{Self-similar cascade with explosion cartoon.}\label{fig2}
\end{center}
\end{figure}
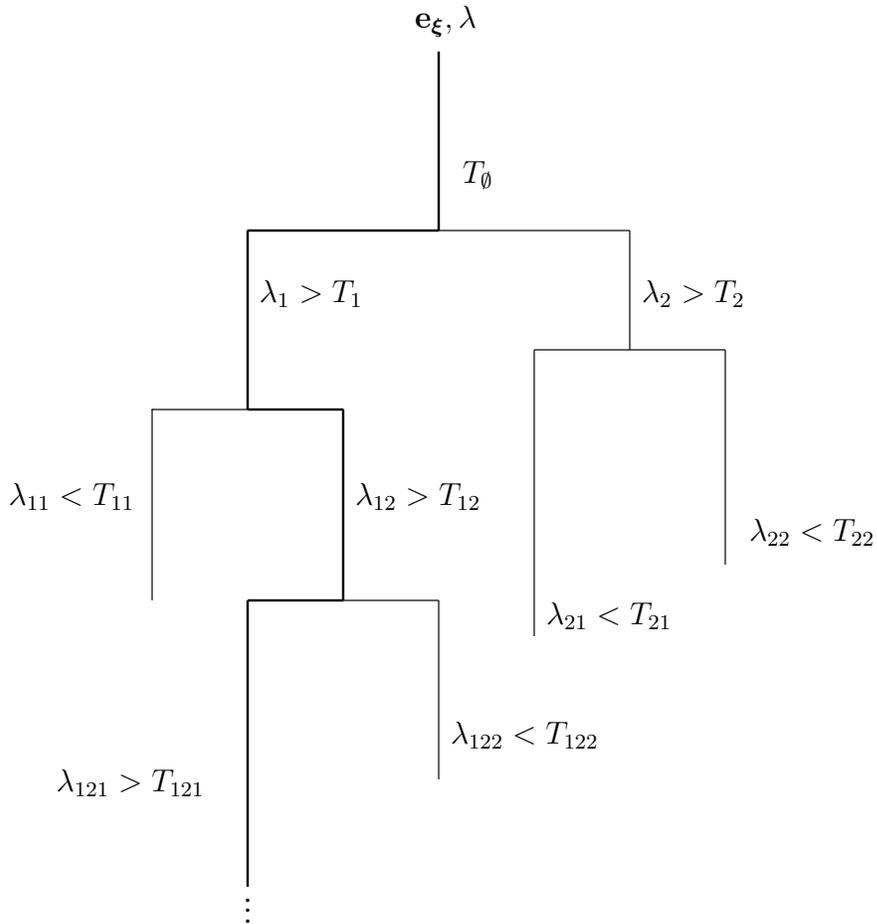

Moreover, comparing the $\zetas_n$ above with $\zeta_n$ -- defined in the context of Definition \ref{expltime}
for the kernel $h_d$ -- we obtain our main result connecting the self-similar and dilogarithmic uniqueness problems.

\begin{thm}\label{thm:main}
The events $[\zeta_n(|\nboldsymbol{\xi}|) > t]$ for a dilogarithmic density
and $[\tilde{\zeta}_n(\mathbf{e}_{\nboldsymbol{\xi}}) > t |\nboldsymbol{\xi}|^2]$ have the same distribution independent on the choice of the initial wavenumber $\nboldsymbol{\xi}$ or initial direction $\mathbf{e}_{\nboldsymbol{\xi}}$, and hence the explosion time $\zeta$ from Definition \ref{expltime} for the dilogarithmic kernel, and the eventual explosion $\zetas$ from Definition \ref{selfsimilarexplosion} have the same distribution, independent of the choice of $\nboldsymbol{\xi}_0$ or $\mathbf{e}_0$.
\end{thm}

\begin{proof}
Recall that when the dilogarithmic kernel is used to determine the distribution of the
branching frequencies, it follows that for any $\langle s \rangle \in \partial \mathcal{V}$ 
$$
R_k = \frac{|\mathbf{W}_{\langle s|k\rangle}|}{|\mathbf{W}_{\langle s|k-1\rangle}|},
~~k\geq 1
$$
is a sequence of iid random variables with density $\mathcal{D}(r)$.
Now, with $\mathbf{W}_{s|0} = \nboldsymbol{\xi},$ one has
\begin{eqnarray*}
\zeta_n(|\nboldsymbol{\xi}|) &=& \inf_{|s|=n}\sum_{j=0}^n|\mathbf{W}_{s|j}|^{-2}T_{s|j} \\
&=& \frac{1}{|\nboldsymbol{\xi}|^2} \inf_{|s|=n} T_{s|0}+\sum_{j=1}^n\prod_{k=1}^{j}
\frac{|\mathbf{W}_{s|k-1}|^{2}}{|\mathbf{W}_{s|k}|^{2}}T_{s|j} \\
&\stackrel{\mathcal{D}}{=}&
\frac{1}{|\nboldsymbol{\xi}|^2}\inf_{| s |=n}
\sum_{j=0}^n T_{{ s|j}} \frac{1}{\prod_{k=0}^{j} |{R}_{k}|^2}
\end{eqnarray*}
where $R_0 =1.$
Thus the we obtain the equality, in distribution, of the events
$
[\zeta_n(|\nboldsymbol{\xi}|) > t].
$
and
$
[\tilde{\zeta}_n(\mathbf{e}_{\nboldsymbol{\xi}}) > t |\nboldsymbol{\xi}|^2].
$
\end{proof}

In analogy with Proposition \ref{calderon}, one has the following

\begin{prop}
\label{integraleqn}
Let
$$
Z(\lambda_0) = 1 + \sum_{n=0}^\infty  \sum_{|v|=n} {\textbf{1}}[T_{ v } < 
\lambda_{ v }].
$$
Define $\tilde{m}(\lambda,k) = {\mathbb{P}}(Z(\lambda) = k).$ 
Let $\tilde{m}(\lambda) = \sum_{k=1}^\infty \tilde{m}(\lambda,k)$.  Then,
\begin{equation}
\label{mildequation}
\tilde{m}(\lambda) = e^{-\lambda} + 
 \int_0^{\lambda} e^{-(\lambda-s)}
\int_0^\pi\int_0^\infty \tilde{m}(r^2s) \tilde{m}((1-2r\cos\theta+r^2)s) 
\tilde{H}(\theta,r) \,d r d \theta d s.
\end{equation}
Moreover, if $\tilde{m}(\lambda) = 1$ is the unique non-negative solution
 then there is no similarity explosion.
\end{prop}
\begin{proof}
Note that
$Z(\lambda_0)$ represents the number of branches of
the self-similar branching process started with horizon $\lambda_0.$ 
Recall that from the definitions of $\lambda_{s}$, $Z$ is independent of the initial
direction. 
Each time the indicator does not vanish, a branching occurs increasing the number of branches
by $1$.  The extra term is to count the initial branch.  

For $k\geq 2,$ condition on the time of the first branching to get, 
\begin{equation}
\label{eq:mlambdak}
\tilde{m}(\lambda,k) = \sum_{j=1}^{k-1} \int_0^{\lambda} e^{-(\lambda-s)}
\int \tilde{m}(|\nboldsymbol{\eta}|^2s,j) \tilde{m}(|\mathbf{e}_0-\nboldsymbol{\eta}|^2 s,k-j) 
H_{d}(\nboldsymbol{\eta}|\mathbf{e}_0) \,d \nboldsymbol{\eta} d s,
\end{equation}
where $\mathbf{e}_0$ is arbitrary.

Let $\tilde{H}(\theta,r)$ denote the average of $H_d$ with respect to the angle $\phi,$
$$
\tilde{H}(\theta,r)
=
\frac{2}{\pi^2}
\frac{\sin\theta}{1-2r\cos\theta+r^2}.
$$
Then, the independence of the above equation with respect to the direction $\mathbf{e}_0$ is better
illustrated in the following;
\begin{equation*}
\tilde{m}(\lambda,k) = \sum_{j=1}^{k-1} \int_0^{\lambda} e^{-(\lambda-s)}
\int_0^\pi\int_0^\infty \tilde{m}(r^2s,j) \tilde{m}((1-2r\cos\theta+r^2)s,k-j) 
\tilde{H}(\theta,r) \, d r d \theta d s.
\end{equation*}
  Summing on $k$ the previous equation and
adding the term corresponding to $k=1,$ one has (\ref{mildequation}).
It is clear that $\tilde{m} \equiv 1$ is a solution of this equation, so non explosion is equivalent to
showing that this is the only  non negative solution that is bounded by $1.$
\end{proof}

While we can not prove that $\tilde{m} \equiv 1$ is the only solution of 
\eqref{mildequation}, we note that the behavior at infinity can be used to 
determined if $\tilde{m}(\lambda) < 1$  on a set of positive measure.
In fact,  if for some $\epsilon >0,$
$\tilde{m}(\lambda)\leq(1-\epsilon)$ on a set $E$ of positive measure, then $\tilde{m}$ is
bounded by a decreasing function. 
Indeed, for any $\lambda>0,$ $0\leq \tilde{m}(\lambda) \leq 1,$ and from \eqref{mildequation}
one has
\begin{eqnarray*}
\tilde{m}(\lambda) &\leq& e^{-\lambda}
+ \int_0^\lambda e^{-(\lambda - s)} \int_0^\infty \tilde{m}(r^2 s)
\mathcal{D}(r) \,d r \\
&<&
\rme^{-\lambda}
+ \int_0^\lambda  \rme^{-s} \,d s
-
\epsilon \int_0^\lambda e^{-s}  \int_E \mathcal{D}(r) \,d r d s \\
&=& 1 - \epsilon \mu(E) (1-e^{-\lambda})
\end{eqnarray*}
where $\mu(E)= \int_E \mathcal{D}(r) \,d r.$

We are now ready to establish one of the main results of the paper. 
Define the finite horizon probability of explosion in a similar way as that of self similar explosion.
To be precise, let 
$$
\tilde{A}(\lambda) = \cap_{n\geq 1} [\tilde{\zeta}_n \leq \lambda].
$$
Then, $\tilde{m}(\lambda) = 1 - {\mathbb{P}}(\tilde{A}(\lambda)),$ and thus 
$$
\limsup_{\lambda\rightarrow\infty} \tilde{m}(\lambda) = \alpha<1
\iff
%\forall \nboldsymbol{\xi}\neq 0, ~~
{\mathbb{P}}(\tilde{A}(\lambda)) > 0.
$$

We then have the following;
\begin{thm}\label{thm:0-1}
The self similar explosion event is a $0,1$ event and independent of the initial direction.
\end{thm}
\begin{proof}
Note that since $\mathbb{P}(\tilde{A}(\lambda))$ is an increasing function of $\lambda$, 
$\tilde{m}(\lambda)$
is decreasing. %(something we have not been able to get directly from the equation).
Let $0\leq \alpha \leq 1$ be defined by 
$\lim_{\lambda\rightarrow\infty} \tilde{m}(\lambda) = \alpha.$ 
Using dominated convergence, one can take limit as $\lambda \rightarrow \infty$ in 
(\ref{mildequation}) to get
$
\alpha = \alpha^2,
$
so $\alpha=0$ or $1.$
The proof is completed, since
${\mathbb{P}}(A_{\mathbf{e}_0})= 1-\alpha$ independent of $\mathbf{e}_0.$ 
\end{proof}

An important consequence of this result is the following corollary.

\begin{cor}
{For any $\nboldsymbol{\xi}\neq 0,$ the explosion event for the Navier-Stokes cascades defined in terms of the dilogarithmic kernel
$H_{d}$ is a 0, 1 event.}
\end{cor}

\begin{proof}
The corollary follows from the equality, in distribution, of the events
$
[\zeta_n(|\nboldsymbol{\xi}|) > t]
$
and
$
[\tilde{\zeta}_n(\mathbf{e}_{\nboldsymbol{\xi}}) > t |\nboldsymbol{\xi}|^2]
$ (See Theorem \ref{thm:main}).
\end{proof}

Similarly, the integral equations \eqref{mequationns} and \eqref{mildequation} can be shown to be
equivalent in the case the Navier-Stokes cascade is defined using the dilogarithmic distribution.
\begin{prop}
\label{calculusexercise}
Let $m(|\nboldsymbol{\xi}|,t)$ be the solution of the integral equation
\begin{equation}
\label{mtotalfourier}
m(|\nboldsymbol{\xi}|,t) = {\rm{e}}^{-|\nboldsymbol{\xi}|^2t} +  |\nboldsymbol{\xi}|^2
\int_0^t {\rm{e}}^{-|\nboldsymbol{\xi}|^2(t-s)}  \int_{{\R^3}} m(|\nboldsymbol{\eta}|,s)
m(|\nboldsymbol{\xi}-\nboldsymbol{\eta}|,s) H_d(\nboldsymbol{\eta}|\nboldsymbol{\xi})
\,d \nboldsymbol{\eta}  d s.
\end{equation}
Then 
\begin{equation}
\label{mareequivalent}
\tilde{m}(\lambda) = m(|\nboldsymbol{\xi}|,\lambda/|\nboldsymbol{\xi}|^2)
\end{equation}
is a solution of
\eqref{mildequation}
Conversely, given a solution $\tilde{m}(\lambda)$ of \eqref{mildequation}, equation
\eqref{mareequivalent} defines a solution of \eqref{mtotalfourier}.
\end{prop}

\begin{proof}
%Due to the spherical symmetry of the kernel $H_\mathcal{D}$ one has that
%$m(|\nboldsymbol{\xi}|,t,k) = P_{\nboldsymbol{\xi}}(Z(\nboldsymbol{\xi},t) = k)$
%as defined in \eqref{def:mtk} satisfies, for $k\geq 2,$
%$$
%m(|\nboldsymbol{\xi}|,t,k) = |\nboldsymbol{\xi}|^2 \sum_{j=1}^{k-1} \int_0^t 
%{\rm{e}}^{-|\nboldsymbol{\xi}|^2 (t-s)} \int_{R^3} 
%m(|\nboldsymbol{\eta}|,s,j) m(|\nboldsymbol{\xi}-\nboldsymbol{\eta}|,s,k-j) 
% H_\mathcal{D}(\nboldsymbol{\eta}|\nboldsymbol{\xi})
%\rmd\nboldsymbol{\eta}  \rmd s.
%$$
Introduce new variables $\nboldsymbol{\eta}=|\nboldsymbol{\xi}|\nboldsymbol{\eta}', s'=|\nboldsymbol{\xi}|^2 s,$ and recall that
$H_d(\nboldsymbol{\eta}|\nboldsymbol{\xi})
\,d\nboldsymbol{\eta}=H_d(\nboldsymbol{\eta}'|{\rm{e}}_{\nboldsymbol{\xi}})
\,d\nboldsymbol{\eta'}$.  Then changing variables in \eqref{mtotalfourier}, and
with $\lambda=|\nboldsymbol{\xi}|^2 t,$ one has, 
$$
m(|\nboldsymbol{\xi}|,\lambda/|\nboldsymbol{\xi}|^2) =
{\rm{e}}^{-\lambda} +
%\sum_{j=1}^{k-1} 
\int_0^{\lambda}%{t  |\nboldsymbol{\xi}|^2}
{\rm{e}}^{-(\lambda-s')}
\int_{\R^3} 
m(|\nboldsymbol{\xi}||\nboldsymbol{\eta}'|,s'/|\nboldsymbol{\xi}|^2) 
m(|\nboldsymbol{\xi}||{\rm{e}}_{\nboldsymbol{\xi}}-\nboldsymbol{\eta}'|,s'/|\nboldsymbol{\xi}|^2) 
 H_d(\nboldsymbol{\eta}'|{\rm{e}}_{\nboldsymbol{\xi}})
\,d\nboldsymbol{\eta'}  d s'.
$$
With $\tilde{m}(\lambda)$ as defined in \eqref{mareequivalent}, one has, dropping primes,
$$
\tilde{m}(\lambda) =
{\rm{e}}^{-\lambda} +
%\sum_{j=1}^{k-1} 
\int_0^{\lambda}%{t  |\nboldsymbol{\xi}|^2}
{\rm{e}}^{(\lambda-s)}
\int_{\R^3} 
\tilde{m}(s |\nboldsymbol{\eta}|^2) 
\tilde{m}(s |{\rm{e}}_{\nboldsymbol{\xi}}-\nboldsymbol{\eta}|^2) 
 H_d(\nboldsymbol{\eta}|{\rm{e}}_{\nboldsymbol{\xi}})
\,d\nboldsymbol{\eta}  d s.
$$
The proof is completed by noting that \eqref{mildequation} is obtained 
from this equation
by integrating the angular variables and, to obtain the converse, reversing the steps.
\end{proof}

%%%
%%%
\subsection{Self Similar cascades and Leray equation}% \eqref{Lerayeqtn}.}

In this subsection we show that the self similar stochastic cascade can be
obtained directly from the Leray forward equations \eqref{Lerayeqtn}

\begin{prop}\label{prop:Leray_eq}
Let $\mathbf{U}(X)$ be a solution of the Leray equation \eqref{Lerayeqtn}, $\hat{\mathbf{U}}$ denote its
Fourier transform.  Then, with $\mathbf{e}_{\nboldsymbol{\xi}}$ a unit vector in $\R^3$ and 
$\lambda>0,$
$$
\mathbf{u}(\mathbf{e}_{\nboldsymbol{\xi}},\lambda) = \lambda\hat{\mathbf{U}}(\sqrt{\lambda} \mathbf{e}_{\nboldsymbol{\xi}}).
$$
%where $\hat
%$$
%-\Delta \mathbf{U} - \frac{1}{2} \mathbf{U} - \frac{1}{2}(\mathbf{X}\cdot\nabla) \mathbf{U} + (\mathbf{U} \cdot \nabla) \mathbf{U} = - \nabla P, \quad
%\nabla\cdot \mathbf{U} =  0.
%$$
%Then $\mathbf{u}(r^2,\mathbf{e}_{\nboldsymbol{\xi}})$ 
satisfies \eqref{mildssns}.  In particular $\mathbf{u}(\mathbf{e}_{\nboldsymbol{\xi}},0) =
\hat{\mathbf{u}}_0(\mathbf{e}_{\nboldsymbol{\xi}})$
\end{prop}

\begin{proof}
Recall that the forward Leray equations are obtained assuming a solution of the Navier-Stokes equations
of the form
$$
\mathbf{u}(\mathbf{x},t) = \frac{1}{\sqrt{t}} \mathbf{U}(\mathbf{x}/\sqrt{t}),
$$
and are given by
\begin{equation}
\label{Learyeqtn}
-\Delta \mathbf{U} - \frac{1}{2} \mathbf{U} - \frac{1}{2}(\mathbf{X}\cdot\nabla) \mathbf{U} + (\mathbf{U} \cdot \nabla) \mathbf{U} = - \nabla P,
~~~~
\nabla\cdot \mathbf{U} =  0.
\end{equation}

Taking Fourier transform and projecting on divergence free vector fields, one gets
$$
(1+|\nboldsymbol{\xi}|^2) \hat{\mathbf{U}} + \frac{1}{2} (\nboldsymbol{\xi}\cdot\nabla)\hat{\mathbf{U}} + (2\pi)^{-3/2} |\nboldsymbol{\xi}|
\int_{\R^3} \hat{\mathbf{U}}(\nboldsymbol{\xi}-\nboldsymbol{\eta}) \odot_{\nboldsymbol{\xi}} \hat{\mathbf{U}}(\nboldsymbol{\eta}) \,d\nboldsymbol{\eta} = 0
$$
Let $\mathbf{e}_{\nboldsymbol{\xi}} = \nboldsymbol{\xi}/|\nboldsymbol{\xi}|$ and define $\mathbf{V}(\mathbf{e}_{\nboldsymbol{\xi}},r) = \hat{\mathbf{U}}(r \mathbf{e}_{\nboldsymbol{\xi}}).$  Since 
$$
\nboldsymbol{\xi}\cdot\nabla \hat{\mathbf{U}} = r \frac{d \mathbf{V}}{d r},
$$
one has, with some abuse of notation,
$$
(1+r^2) \mathbf{V}+ \frac{1}{2}  r \frac{d \mathbf{V}}{d r} + (2\pi)^{-3/2} r
\int _{\R^3}\hat{U}(\nboldsymbol{\xi}-\nboldsymbol{\eta}) \odot_{\nboldsymbol{\xi}} \hat{\mathbf{U}}(\nboldsymbol{\eta}) \,d\nboldsymbol{\eta} = 0.
$$
Multiplying the equation by $2 r e^{r^2},$ one obtains
$$
\frac{d}{d r} (r^2 e^{r^2} \mathbf{V}) =
-(2\pi)^{-3/2} 2r^2 e^{r^2}
\int _{\R^3}\hat{\mathbf{U}}(\nboldsymbol{\xi}-\nboldsymbol{\eta}) \odot_{\nboldsymbol{\xi}} \hat{\mathbf{U}}(\nboldsymbol{\eta}) \,d\nboldsymbol{\eta}.
$$
Let $\tilde{\mathbf{V}}(\mathbf{e},r) = r^2 \mathbf{V}(\mathbf{e},r).$  Then
\begin{equation}
\label{auxiliar}
\frac{d}{d r} (e^{r^2} \tilde{\mathbf{V}}) = 
-(2\pi)^{-3/2} 2r e^{r^2}
\int _{\R^3}\tilde{\mathbf{V}}(\mathbf{e}_{r\mathbf{e}-\nboldsymbol{\eta}},|r\mathbf{e}-\nboldsymbol{\eta}|) \odot_{\nboldsymbol{\xi}} \tilde{\mathbf{V}}(\mathbf{e}_{\nboldsymbol{\eta}},|\nboldsymbol{\eta}|) \frac{r}{|r\mathbf{e}-\nboldsymbol{\eta}|^2|\nboldsymbol{\eta}|^2}\,d\nboldsymbol{\eta}.
\end{equation}
Note that one factor of $r$ is used to get, up to a constant, $H_{d}(\nboldsymbol{\eta}|r\mathbf{e}).$  

One may easily check that 
$$
\lim_{r\rightarrow 0} \tilde{ \mathbf{V}}( \mathbf{e},r) = \hat{\mathbf{u}}_0(\mathbf{e}).
$$
Indeed,  since
$\hu(t,\nboldsymbol{\xi})=t\hU(\sqrt{t}\nboldsymbol{\xi})$,  for $\nboldsymbol{\xi}=\exi$ we have:
\[\hu_0(\exi)=\lim\limits_{t\to 0}\hu(t,\exi)=\lim\limits_{t\to 0}t\,\hU(\sqrt{t},\xi)=\lim\limits_{t\to 0}t\mathbf{V}(\exi,\sqrt{t})=\lim\limits_{t\to 0}\tilde{\mathbf{V}}(\exi,\sqrt{t}).
\]

Integrating equation \eqref{auxiliar}, and accounting for the constant to get $H_d$, we obtain
$$
e^{r^2} \tilde{\mathbf{V}}(\mathbf{e},r) = \hat{\mathbf{u}}_0(\mathbf{e}) - (\pi/2)^{3/2}
\int_0^r 2s e^{s^2} 
\int_{\R^3} \tilde{\mathbf{V}}(\mathbf{e}_{s\mathbf{e}-\nboldsymbol{\eta}},|s\mathbf{e}-\nboldsymbol{\eta}|) \odot_{\nboldsymbol{\xi}} \tilde{\mathbf{V}}(\mathbf{e}_{\nboldsymbol{\eta}},|\nboldsymbol{\eta}|) 
H_d(\nboldsymbol{\eta}|s\mathbf{e}) \,d\nboldsymbol{\eta} \,d s.
$$
With the change of variables $\nboldsymbol{\eta} = s\nboldsymbol{\eta}',$ and noting that $\mathbf{e}_{s\mathbf{e}-\nboldsymbol{\eta}}
=\mathbf{e}_{\mathbf{e}-\nboldsymbol{\eta}'}$ and that 
$H_d(\nboldsymbol{\eta}|s\mathbf{e}) \,d \nboldsymbol{\eta} = 
H_d(\nboldsymbol{\eta}'|\mathbf{e}) \,d\nboldsymbol{\eta}',$
we have, dropping primes
$$
e^{r^2} \tilde{\mathbf{V}}(\mathbf{e},r) = \hat{\mathbf{u}}_0(\mathbf{e}) - (\pi/2)^{3/2}
\int_0^r 2s e^{s^2} 
\int_{\R^3} \tilde{\mathbf{V}}(\mathbf{e}_{\mathbf{e}-\nboldsymbol{\eta}},s|\mathbf{e}-\nboldsymbol{\eta}|) \odot_{\nboldsymbol{\xi}} \tilde{\mathbf{V}}(\mathbf{e}_{\nboldsymbol{\eta}},s|\nboldsymbol{\eta}|) 
H_d(\nboldsymbol{\eta}|\mathbf{e}) \,d\nboldsymbol{\eta} \,d s.
$$
Let $t=s^2$ to get
$$
\tilde{\mathbf{V}}(\mathbf{e},r) = e^{-r^2} \hat{\mathbf{u}}_0(\mathbf{e}) - (\pi/2)^{3/2}
\int_0^{r^2}  e^{-(r^2-t)} 
\int _{\R^3}\tilde{\mathbf{V}}(\mathbf{e}_{\mathbf{e}-\nboldsymbol{\eta}},\sqrt{t}|\mathbf{e}-\nboldsymbol{\eta}|) \odot_{\nboldsymbol{\xi}} \tilde{\mathbf{V}}
(\mathbf{e}_{\nboldsymbol{\eta}},\sqrt{t}|\nboldsymbol{\eta}|) 
H_d(\nboldsymbol{\eta}|\mathbf{e}) \,d\nboldsymbol{\eta} \,d t.
$$
%Define now $\mathbf{u}(r^2,\mathbf{e}) = \mathbf{W}(r,\mathbf{e})$ to get
%\begin{equation}
%\label{FourierLeray}
%\mathbf{u}(r^2,\mathbf{e}) = e^{-r^2} \mathbf{\mathbf{u}}_0(\mathbf{e}) - (\pi/2)^{3/2}
%\int_0^{r^2}  e^{-(r^2-t)} 
%\int \mathbf{u}(t|\mathbf{e}-\nboldsymbol{\eta}|^2,\mathbf{e}_{\mathbf{e}-\nboldsymbol{\eta}}) \odot_{\nboldsymbol{\xi}} \mathbf{u}(t|\nboldsymbol{\eta}|^2,\mathbf{e}_{\nboldsymbol{\eta}}) H(\mathbf{e},\nboldsymbol{\eta}) \,d\nboldsymbol{\eta} \,d t
%\end{equation}
%which is \eqref{mildssns} with a slightly different notation.

The proof is completed setting $\lambda = r^2$ and defining $\mathbf{u}(\mathbf{e},r^2) = 
\tilde{\mathbf{V}}(\mathbf{e},r)$.

\end{proof}

%{I am still wandering what to do with this remark}.

\begin{remark}  As an aside, one may note that the choice of the scaling
parameter $r$ is completely arbitrary.
Corresponding to the choices $r = 1/\sqrt{t}$ made by Leray, and
say, $r = 1/|\mathbf{x}|$, respectively, let 
$\mathbf{u}_1(\mathbf{x},t)= (1/\sqrt{t}) \mathbf{U}(\mathbf{x}/\sqrt{t}),$  and
 $\mathbf{u}_2(x,t) = (1/|\mathbf{x}|) \mathbf{V}(\mathbf{x}/|\mathbf{x}|, t/|\mathbf{x}|^2)$
Let's note that $\mathbf{U}$ and $\mathbf{V}$ can be related by an application of the Kelvin transform
${\mathcal{T}}_1$
with respect to the unit sphere in $\R^3.$   To see this, recall that 
$\mathcal{T}_a [\mathbf{u}(\mathbf{y})] \equiv (a/|\mathbf{y}|) \mathbf{u}((a^2/|\mathbf{y}|^2) \mathbf{y}),$
defines the Kelvin transform of $\mathbf{u}$ 
with respects to the sphere of radius $a$. Now, letting $\mathbf{X}=\mathbf{x}/\sqrt{t}$
one has
\begin{eqnarray*}
\mathcal{T}_1 [(1/\sqrt{t}) \mathbf{U}(\mathbf{X})] &=&
\frac{1}{\sqrt{t}|\mathbf{X}|} \mathbf{U}(\mathbf{X}/|\mathbf{X}|^2) =\\
\frac{1}{|\mathbf{x}|} \mathbf{U}(\mathbf{x}\sqrt{t}/|\mathbf{x}|^2) &=&
\frac{1}{|\mathbf{x}|} \tilde{\mathbf{U}}(\mathbf{x}/|\mathbf{x}|,\sqrt{t}/|\mathbf{x}|) 
\equiv \frac{1}{|\mathbf{x}|} \mathbf{V}(\mathbf{x}/|\mathbf{x}|,t/|\mathbf{x}|^2).
\end{eqnarray*}
\end{remark}

\section{Conclusions and Further Directions}\label{sec4}

The primary goal of this article was to
precisely formulate a notion of symmetry breaking for the
three-dimensional incompressible Navier-Stokes equations,
and to provide an approach to the resulting symmetry 
breaking vs or not dichotomy.  The notion
that is introduced builds on a variant of classic scaling and
self-similarity ideas of Leray \cite{JL1934}. Namely, symmetry 
breaking is defined as a phenomena in which one has
uniqueness of self-similar solutions, but non-uniqueness
of general solutions.  The approach is derived from
a stochastic cascade representation (NSC) of the Navier-Stokes  equations 
introduced by Le Jan an Sznitman \cite{YLAS1997}, together with a
corresponding
development of a cascade representation (SSC) of (mild)
self-similar solutions.  
The essence of the approach is to exploit a notion of
branching process explosion as a surrogate to non-uniqueness.
A branching random walk cascade,
namely the binary branching dilogarithmic random walk on $(0,\infty)$
viewed as a multiplicative group, is obtained as a common
element of both representations for comparison.
A principle result was the equivalence of the explosion
phenomena for (NSC) and (SSC).  In addition it is shown
that the explosion criteria is critical in the sense of scaling,
and a zero-one law is established for the explosion event.

It remains to firm up the precise connection between explosion and non-uniqueness.  A related semilinear 
pseudo-differential equation of Proposition \ref{prop:PsiDE} and an integral equation of Proposition \ref{integraleqn}
can be associated with the branching numbers in such
a way that uniqueness of solutions to either in an appropriate
space is shown to be equivalent to non-explosion.  In fact,
although not obvious, as shown by Proposition \ref{calculusexercise},
 the two equations are equivalent.
However
the yet unproven connection between explosion criteria and
uniqueness is expected to be that 
non-explosion corresponds to the uniqueness of mild
solutions represented by (NSC) and (SSC), respectively.
Proving this in appropriate function spaces is a substantial challenge to the overall approach.   
Assuming that this will be achievable, the surrogate results will prove that the equations are in fact not symmetry breaking.

\section{Appendix: Bessel and Dilogarithmic Markov Chains \& Explosion}\label{appendix}

This appendix records some general approaches to the explosion problem that may eventually prove useful as
we learn more about the dilogarithmic branching random
walk.  In fact,  we are able to demonstrate their effectiveness
when applied to the simpler case of the Bessel kernel, which
we show to be non-explosive.  At a heuristic level, it is the
mean reverting property (\ref{meanreversion}) that makes
the Bessel kernel simpler to analyze. 

The first approach to explosion exploits the monotonicity in the sequence $\{\zeta_n\}.$

\begin{prop}
\label{prop:monotonicity}
Let $\zeta$ be as in Definition \ref{expltime}, 
and assume that for some $\lambda>0,$
$$2^n\mathbb{E}_{|\nboldsymbol{\xi}|}\prod_{j=1}^n\frac{\lambda}{\lambda + |\mathbf{W}_j|^2} 
\to 0, \quad \text{as }\ n\to\infty.$$
Then
$
\mathbb{P}([\zeta = \infty]) = 0.
$
\end{prop}

\begin{proof}
To prove non-explosion it suffices to show that for any $B > 0$,
$$P_{|\nboldsymbol{\xi}|}(\zeta_n > B \ \rm{eventually}) = 1,$$
or equivalently, that
$$P_{|\nboldsymbol{\xi}|}([\zeta < B]) = 
P_{|\nboldsymbol{\xi}|}(\cap_{n=1}^\infty[\zeta_n  < B]) \equiv \lim_{n\to\infty}P_{|\nboldsymbol{\xi}|}(\zeta_n < B) = 0$$
where we have used the monotonicity of the sequence $\{\zeta_n\}.$
For the latter observe that,
\begin{eqnarray*}
P_{|\nboldsymbol{\xi}|}(\zeta_n < B) &=& P_{|\nboldsymbol{\xi}|}(\min_{|s|= n}\sum_{j=1}^n|W_{s|j}|^{-2}T_{s|j} < B)
\nonumber\\
&\le& 2^nP_{|\nboldsymbol{\xi}|}(\sum_{j=1}^n|W_{1|j}|^{-2}T_{1|j} < B)\nonumber\\
&=& 2^nP_{|\nboldsymbol{\xi}|}(e^{-\lambda\sum_{j=1}^n|W_{1|j}|^{-2}T_{1|j}} > e^{-\lambda B})\nonumber
\end{eqnarray*}
for any $\lambda>0,$ where $1|j = (1,1,\dots,1)$ is on the fixed, but 
otherwise arbitrary,
tree path $(1,1,\dots)$. By the Markov inequality, one has
\begin{eqnarray*}
P_{|\nboldsymbol{\xi}|}(\zeta_n < B) 
&\le& 2^n\mathbb{E}_{|\nboldsymbol{\xi}|}e^{-\lambda\sum_{j=1}^n|W_{1|j}|^{-2}T_{1|j}} e^{\lambda B}\nonumber\\
&=&2^ne^{\lambda B}\mathbb{E}_{|\nboldsymbol{\xi}|}\prod_{j=1}^n\frac{\lambda}{ \lambda + |W_{1|j}|^2},
\end{eqnarray*}
which converges to $0$ as $n\to \infty.$
%Thus it is sufficient to show that 
%$$2^n\mathbb{E}_{|\nboldsymbol{\xi}|}\prod_{j=1}^n{\lambda\over \lambda + |W_j|^2} \to 0, \quad \rm{as}\ n\to\infty.$$
%\alert{Notice that by exploiting monotonicity we have reduced the convergence of an infinite series to the condition that the terms form a null sequence.}
\end{proof}

As an illustration of this methodology we provide
a proof of Theorem \ref{th:nonexplosionBessel} below,
establishing that the branching Markov Chain defined using the Bessel
kernel $h_b(\nboldsymbol{\xi})$ to determine the distribution of the branching Fourier frequencies does
not explode. 
The mean reversion property (\ref{meanreversion}) provides some indication as to why one may expect the corresponding branching Markov chain to be non-explosive, as will be shown is indeed the case.  Namely,

\begin{thm} 
\label{th:nonexplosionBessel}
The explosion horizon is almost surely infinite for the Bessel Markov chain.
\end{thm}

 For the proof we first note the following more refined
property of the Bessel Markov chain.

\begin{lemma}
\label{lm:expectation}
Assume that $W$ is a non negative random variable with probability density
$$p_u(w) = 
\left\{
\begin{array}{ll}
\frac{1}{u}e^{-2w} (e^{2u} -1) & {\rm{for}}~~w\geq u, u > 0 \\
\frac{1}{u}(1-e^{-2w}) & {\rm{for}}~~~0\leq w \leq u, u > 0.
\end{array}
\right.
$$
where $u$ is an arbitrary positive constant.
Then
$$\mathbb{E}_u\frac{\lambda}{\lambda + W^2} \le \pi\sqrt{\lambda}, \quad \forall\, u, \lambda > 0.$$
\end{lemma}
\begin{proof}
Use integration by parts to note that
$$
\int
 \frac{\lambda}{\lambda + w^2}  e^{-2w} ~{\rm{d}}w
 =
 \sqrt{\lambda} \arctan(w/{\sqrt\lambda}) e^{-2w} 
 +
 2\sqrt{\lambda} \int \arctan(w/{\sqrt\lambda}) e^{-2w} ~{\rm{d}}w.
$$
One has
\begin{align*}
\mathbb{E}_u\frac{\lambda}{\lambda + w^2} =&
\int_0^\infty\frac{\lambda}{\lambda + w^2} p_u(w) ~ {\rm{d}}w \\
=&
\frac{1}{u}\int_0^u \frac{\lambda}{\lambda + w^2} ~{\rm{d}}w
-\frac{1}{u}\int_0^\infty \frac{\lambda}{\lambda + w^2}  e^{-2w} ~{\rm{d}}w 
+
\frac{1}{u}\int_u^\infty \frac{\lambda}{\lambda + w^2}  e^{-2(w-u)} ~{\rm{d}}w \\
=& 
\frac{1}{u} \sqrt{\lambda} \arctan(u/{\sqrt\lambda}) 
-\frac{1}{u} 2\sqrt{\lambda}
\int_0^\infty \arctan(w/\sqrt{\lambda}) e^{-2w} ~{\rm{d}}w \\
&-\frac{1}{u}\sqrt{\lambda} \arctan(u/\sqrt{\lambda}) 
+
\frac{1}{u} 2 \sqrt{\lambda} \int_u^\infty \arctan(w/\sqrt{\lambda}) e^{-2(w-u)} ~{\rm{d}}w \\
=&
\frac{2 \sqrt{\lambda}}{u}
\int_u^\infty
\arctan(w/\sqrt{\lambda}) (e^{-2(w-u)}  - e^{-2w}) ~{\rm{d}}w
-
\frac{2 \sqrt{\lambda}}{u}
\int_0^u
\arctan(w/\sqrt{\lambda}) e^{-2w} ~{\rm{d}}w \\
\leq &
\frac{\pi \sqrt{\lambda}}{u}
\int_u^\infty
 (e^{-2(w-u)}  - e^{-2w}) ~{\rm{d}}w
 =
 \pi \sqrt{\lambda} 
 \frac{1}{2 u}(1-e^{-2u}).
\end{align*}
The result follows by noting that $(1-e^{-x})/x \leq 1$ for any $x.$
\end{proof}

\noindent
{\textit{Proof of Theorem \ref{th:nonexplosionBessel}}}

Note that successive use of conditional expectations on $\mathcal{F}_{j},$ 
the sigma field generated by the branching process
up to the $j^{\rm{th}}$ branching event, and Lemma
\ref{lm:expectation} one has
\begin{eqnarray*}
\mathbb{E}_{|\nboldsymbol{\xi}|}\prod_{j=1}^n\frac{\lambda}{ \lambda + |\mathbf{W}_j|^2} &=&
\mathbb{E}_{|\nboldsymbol{\xi}|} \left[ \prod_{j=1}^{n-1}\frac{\lambda}{ \lambda + |\mathbf{W}_j|^2}
\mathbb{E}_{|\mathbf{W}_{n-1}|} (\frac{\lambda}{\lambda + |\mathbf{W}_n|^2})\right] \\
&\leq&
(\pi\sqrt{\lambda}) ~ 
\mathbb{E}_{|\nboldsymbol{\xi}|} \left[ \prod_{j=1}^{n-1}\frac{\lambda}{ \lambda + |\mathbf{W}_j|^2}\right]\\
&\leq&
(\pi\sqrt{\lambda})^n 
\end{eqnarray*}

The theorem follows applying Proposition \ref{prop:monotonicity} with
$\lambda<1/(2\pi)^2.$
\qed

%\begin{remark}
%The factor $2^n$ may be replace by $(2p)^n$ for $1/2 < p %\le 1$ for the case in which the branching process
%involves deaths with probability $1-p$, as can be %incorporated in the case for the Le Jan-Sznitman cascade 
%representation with forcing.
%\end{remark}

Use of the monotonicity approach is less transparent for
analysis of the
dilogarithmic explosion problem.  Another
 approach is generally possible
that builds on a variant of the Biggins-Kingman-Hammersley (BKH), e.g., see \cites{JB1977, JB1997, JB2010, EBBD1999},
computation of the speed of the
leftmost particle for additive branching random walks in terms of multiplicative branching
random walk.  It is potentially applicable to the dilogarthmic kernel
precisely because for any path $s\in\{1,2\}^\infty$, 
\begin{equation}
|\mathbf{W}_{s|n}| = |\nboldsymbol{\xi}|\prod_{j=1}^n\frac{|\mathbf{W}_{s|j}|}{ |\mathbf{W}_{s|j-1}|}, n = 1,2,\dots,
\end{equation}
and the ratios are i.i.d. That is,
 for any path $s\in\{1,2\}^\infty$, the sequence
$\{|\mathbf{W}_{s|j}| : j = 0,1,\dots\}$ is a random walk on the
multiplicative group $(0,\infty)$ starting at $|\mathbf{W}_{s|0}| = |\nboldsymbol{\xi}|$. That is, for the dilogarithmic kernel the branching
Markov chain is in fact a
branching random walk on the multiplicative group
$(0,\infty)$.

First let us recall the general heuristic underlying (BKH) speed calculations on the additive group of real numbers:
Suppose that $\{S_n: n = 0,1,2,\dots\}$ is an additive random walk on $\mathbb{R}$
with mean zero and starting at zero.   Then by the weak law of large numbers $P(S_n < nc)\to 0$ as
$n\to\infty$ for any $c < 0$.   
Let $m(\theta) = \mathbb{E}e^{\theta S_1}$ and consider
 the following large deviation inequality
\begin{eqnarray}
m(\theta)^n &=& \mathbb{E}e^{\theta S_n}\nonumber\\
&\ge& e^{n\theta c}P(S_n > nc).
\end{eqnarray}
Thus,
\begin{equation}
P(S_n > nc) \le \exp\{-n(\theta c - \ln m(\theta))\},
\end{equation}
and in particular,
\begin{equation}
P(S_n > nc) \le \exp\{-n\sup_{c< 0}(\theta c-\ln m(\theta))\} = e^{-nI(c)},
\end{equation}
where
$I(c) = \sup_{c< 0}(\theta c-\ln m(\theta))$ is the Legendre transform of
$\ln m(\theta)$ at $c$.   The Cramer-Chernoff theorem provides general
conditions for which 
$$\lim_{n\to\infty}\frac{\ln P(S_n > nc)}{ n} = -I(c).$$
To apply this to the computation of the speed of the left-most particle of 
a branching random walk one reasons as follows:  At the $n$-th generation
the expected number of particles located to the left of $c < 0$ is 
$2^ne^{-nI(c)}$ for large $n$.  Thus  the extremal speed is given by
 $\gamma = c$
such that $2^ne^{-nI(c)} \approx 1$.  The (BKH) theorem confirms this. 
For the calculations involved here it is actually enough to calculate
a lower bound on the speed.  

This principle translates to the multiplicative group
as follows: 
\begin{prop}  
\label{thm:leftmostparticlespeed}
Consider a binary branching random walk on the multiplicative group 
$(0,\infty)$.
That is, the $n$-th generation particle position for the genealogy $s = (s_1,\dots,s_n)\in\{1,2\}^n$ is given by the product $\prod_{j=1}^nY_{s|j}$, where $(Y_{v*1},Y_{v*2})$'s are i.i.d random vectors with positive components.  Then
\begin{equation*}
\lim_{n\to\infty}\min_{|s|=n}(\prod_{j=1}^nY_{s|j})^{1/ n} = e^\gamma,
\end{equation*}
where $\gamma$ is the speed of the additive branching random walk with 
displacements $\ln Y_v$, $v\in\cup_{n=1}^\infty\{1,2\}^n$.
\end{prop}

\begin{proof}
Simply write $\prod_{j=1}^nY_{s|j} = \exp\{\sum_{j=1}^n\ln Y_{s|j}\}$.
Then 
$$\lim_{n\to\infty}\min_{|s|=n}(\prod_{j=1}^nY_{s|j})^{1/n}
= \exp\{\lim_{n\to\infty}\min_{|s|=n}\frac{\sum_{j=1}^n\ln Y_{s|j}}{ n}\}.$$
The assertion follows from the (BKH) theory since the exponential 
function is a continuous bijection.
 \end{proof}

This now provides the following approach to proving
non-explosion by exploiting the theory of the
speed of  extremal (leftmost) particles in branching random walks.  Namely,

\begin{prop}
\label{prop:BC}
For $s\in\cup_{n=1}^\infty\{1,2\}^n,$ let $\{T_{s}\}$ be i.i.d.
 mean one exponentially distributed random variables independent of random
 variables in $\mathbb{R}^3,$ and independent of 
 $\{{\mathbf{W}_s}\}.$ 
 Assume that
\begin{equation}
\liminf_{n\to\infty}\min_{|s|=n}\big(\prod_{j=1}^n|\mathbf{W}_{s|j}|^{-1}\big)^{1/n} > 0.
\end{equation}
Then 
 $
\mathbb{P}([\zeta = \infty]) = 0.
$
\end{prop}

\begin{proof}

 In view of the Borel-Cantelli lemma one has that 
 \begin{equation}
\sum_{n=1}^\infty P_{\nboldsymbol{\xi}}(\min_{|s|=n}\sum_{j=1}^n|\mathbf{W}_{s|j}|^{-2}T_{s|j} \le M) < \infty, \ \text{for each} \ M > 0
 \end{equation}
 is a sufficient condition for explosion to be a null event.   Observe that for arbitrary 
 $M, \lambda > 0$,
 \begin{eqnarray}
 P_{\nboldsymbol{\xi}}(\min_{|s|=n}\sum_{j=1}^n|\mathbf{W}_{s|j}|^{-2}T_{s|j} \le M) 
 &=& P(e^{-\lambda\min_{|s|=n}\sum_{j=1}^n|\mathbf{W}_{s|j}|^{-2}T_{s|j}} \ge e^{-\lambda M})\nonumber\\
 &\le& e^{\lambda M}\mathbb{E}e^{-\lambda\min_{|s|=n}\sum_{j=1}^n|\mathbf{W}_{s|j}|^{-2}T_{s|j}}.
 \end{eqnarray}
 Also, since the mean of squares is larger than the square of the mean, and since
the arithmetic mean is larger than the geometric mean, one has
\begin{eqnarray}
\label{lowbound}
\min_{|s|=n}\sum_{j=1}^n|\mathbf{W}_{s|j}|^{-2}T_{s|j} &=& \min_{|s|=n}\sum_{j=1}^n\big(|\mathbf{W}_{s|j}|^{-1}\sqrt{T_{s|j}}\big)^2\nonumber\\
&\ge& n\min_{|s|=n}\big(\frac{1}{ n}\sum_{j=1}^n|\mathbf{W}_{s|j}|^{-1}\sqrt{T_{s|j}}\big)^2\nonumber\\
&\ge& n\min_{|s|=n}\big(\prod_{j=1}^n|\mathbf{W}_{s|j}|^{-1}\sqrt{T_{s|j}}\big)^{2/ n}.
\end{eqnarray}
So the problem is reduced to showing that
\begin{equation}
\liminf_{n\to\infty}\min_{|s|=n}\big(\prod_{j=1}^n|\mathbf{W}_{s|j}|^{-1}\sqrt{T_{s|j}}\big)^{1/n} > 0.
\end{equation}
The  indicated (positive) lower bound is possibly infinite.
Since 
$$\min_{|s|=n}\big(\prod_{j=1}^n|\mathbf{W}_{s|j}|^{-1}\sqrt{T_{s|j}}\big) \ge \min_{|s|=n}\big(\prod_{j=1}^n|\mathbf{W}_{s|j}|^{-1}\big)\min_{|s|=n}\big(\prod_{j=1}^n\sqrt{T_{s|j}}\big)$$
the two multiplicative factors can be treated separately.  Moreover, the factor of $n$ may in (\ref{lowbound}) may
be included in either of these factors.   The next calculation shows that it is most effectively
assigned to the first factor.

Namely, let $\gamma_1$ be the speed for $\prod_{j=1}^n\sqrt{T_{s|j}}$. Then $\gamma_1$ is
directly computable from the above variant Proposition \ref{thm:leftmostparticlespeed} 
on (BKH).  
However it is sufficient to bound $\gamma_1$  away from zero.  Accordingly one has the following simple estimate.

For $M > 0$ and $u > 0$ one has
\begin{eqnarray}
\sum_{n=1}^\infty P(\min_{|s|=n}\prod_{j=1}^n\sqrt{T_{s|j}} < M) &\le& \sum_{n=1}^\infty 2^n P(\prod_{j=1}^n
\sqrt{T_{1|j}} < M)\nonumber\\
&=&  \sum_{n=1}^\infty 2^n P(\,\frac{1}{ 2n}\sum_{j=1}^n\ln(T_{1|j}) < \ln M\,)\nonumber\\
&=&  \sum_{n=1}^\infty 2^n P(\,e^{-\frac{u}{ 2}\sum_{j=1}^n\ln(T_{1|j})}  > e^{-un\ln M}\,)\nonumber\\
&\le&  \sum_{n=1}^\infty e^{n(\ln 2 + u\ln M + \ln \Gamma(1-\frac{u}{ 2}))}.
\end{eqnarray}
This series converges for $\Gamma(1-\frac{u}{ 2}) <\frac{1}{ 2}M^{-u}$.  Thus, taking $u = 1$, 
the series converges for any $M <\frac{1}{ 2\sqrt{\pi}}.$  It now follows from the Borel-Cantelli lemma
that  $$\gamma_1 >\frac{1}{2\sqrt{\pi}} > 0.$$
\end{proof}

Regardless of the approach taken, the resolution of the
explosion problem clearly involves a thorough understanding
of the dilogarithmic branching random walk and its properties.  We conclude this appendix with a few 
properties that may prove useful to this end and, at least,
provide some insight into the technical nature of the problem
in a future analysis. 

To the best of our knowledge, the dilogarithmic random walk
is introduced in the present article for the first time.  However an extensive treatment of the
dilogarithm function, its properties and a selection of other applications in both physics and mathematics, is available in \cite{AK1995}.

For the purposes of this article,  let us note
 the invariance (multiplicative group symmetry about the identity) of 
the distribution of the ratios, one has
\begin{equation}
 P_{\nboldsymbol{\xi}}(\frac{|\mathbf{W}_{s|j+1}|}{ |\mathbf{W}_{s|j}|} \le r) = 
\begin{cases}
2\pi^{-2}[\Li_2(r) - \Li_2(-r)] &\hbox{if }\ \  0 < r < 1 \\
1-2\pi^{-2}[\Li_2(\frac{1}{ r}) - \Li_2(-\frac{1}{ r})], &\hbox{if }\  \ r > 1.  
\end{cases}
\end{equation}

On the other hand, it is a rather direct calculation to check
that
\begin{prop}
\begin{eqnarray*}
\mathbb{E}_{\nboldsymbol{\xi}}\frac{|\mathbf{W}_{s|1}|}{ |\mathbf{W}_{s|0}|} &=& \infty, \quad \forall \ 0\neq\nboldsymbol{\xi}\in\mathbb{R}^3\\
\mathbb{E}_{\nboldsymbol{\xi}}\ln\frac{|\mathbf{W}_{s|1}|}{ |\mathbf{W}_{s|0}|} &=& 0, \quad \forall \ 0\neq\nboldsymbol{\xi}\in\mathbb{R}^3\\
\mathbb{E}_{\nboldsymbol{\xi}}\big(\ln\frac{|\mathbf{W}_{s|1}|}{ |\mathbf{W}_{s|0}|}\big)^m &<& \infty, \quad \forall m \ge 1, \ 0\neq\nboldsymbol{\xi}\in\mathbb{R}^3
\end{eqnarray*}
In particular, $\ln|\mathbf{W}_{s|n}| = \ln|\nboldsymbol{\xi}| + \sum_{j=1}^n\ln\frac{|\mathbf{W}_{s|j}|}{ |\mathbf{W}_{s|j-1}|}, n = 1,2,\dots,$ is a martingale.
\end{prop}

As a consequence one has the following

\begin{cor}  The dilogarithmic random walk is 1-neighborhood recurrent in the sense that for fixed but arbitrary $s\in\{1,2\}^\infty,$ for each
$\delta > 1$
\begin{equation*}
P(|\mathbf{W}_{s|n}| < 1 + \delta\   i.o.) = 1
\end{equation*}
In particular, the path-wise explosion times are a.s. infinite for each path $s$.
\end{cor}

\begin{proof}
By the  Chung-Fuchs theorem it follows that $\ln|\mathbf{W}_{s|n}|$ is 0-neighborhood recurrent. That is, given
$\epsilon > 0$, 
\begin{equation*}
P(|\ln |\mathbf{W}_{s|n}| < \epsilon\  i.o.) = P(e^{-\epsilon} < |\mathbf{W}_{s|n}| < e^{\epsilon}\  i.o.) = 1.
\end{equation*}
The assertion follows by taking $\epsilon = \ln(1+\delta).$
\end{proof}

\begin{cor} 
\begin{equation*}
\mathbb{E} \frac{a^2R^2}{ a^2R^2+ \theta} = \frac{2}{\pi}\arctan(\frac{a}{\sqrt{\theta}}), \ \ \theta > 0, a\in\mathbb{R}.
\end{equation*}
\end{cor}

\begin{proof}
Define, 
$g(x) = \mathbb{E}\frac{R^2}{R^2 + x^2}.$
Justify differentiation under the integral to get, with $c=2/\pi^2,$
\begin{eqnarray}
\nonumber
g'(x) &=& c x \int_0^\infty  \frac{ -2 r}{(r^2+x^2)^2} \ln|\frac{1+r}{1-r}| dr \\
\nonumber
&=& c x \int_0^\infty \frac{d}{dr}\left((r^2+x^2)^{-1}\right) \ln|\frac{1+r}{1-r}| dr \\
\nonumber
&=& cx \lim_{\epsilon \to 0^+} \left[\int_0^{1-\epsilon} + \lim_{M\to \infty} \int_{1+\epsilon}^M\right]
\frac{d}{dr}\left((r^2+x^2)^{-1}\right) \ln|\frac{1+r}{1-r}| dr \\
\nonumber
&=&
cx \lim_{\epsilon\to 0^+} \left(\frac{1}{(1-\epsilon)^2+x^2} \ln(\frac{2-\epsilon}{\epsilon})
- \frac{1}{(1+\epsilon)^2 + x^2} \ln(\frac{2+\epsilon}{\epsilon})\right) \\
\nonumber
&& 
\!\!\!\!\!\!
\!\!\!\!\!\!
\!\!\!\!\!\!
-cx\left(
\lim_{\epsilon \to 0^+} 
\int_0^{1-\epsilon} \frac{1}{r^2+x^2} \left(\frac{1}{1+r} + \frac{1}{1-r}\right) dr
+ \lim_{\epsilon \to 0^+} \lim_{M\to \infty} \int_{1+\epsilon}^M
\frac{1}{r^2+x^2} \left(\frac{1}{1+r} - \frac{1}{r-1}\right) dr  \right) \\
\nonumber
&=&
-c \frac{2}{1+x^2}
%\left(
\lim_{\epsilon \to 0^+} 
\left(\arctan(r/x) + (x/2) (\ln(\frac{1+r}{1-r}) \right) \Big{|}_{r=0}^{r=1-\epsilon} \\
%\right)
\nonumber
&&
\nonumber
-c \frac{2}{1+x^2}
\lim_{\epsilon \to 0^+} \lim_{M\to \infty}
\left(\arctan(r/x) + (x/2) (\ln(\frac{1+r}{r-1}) \right) \Big{|}_{r=1+\epsilon}^M \\
&=&
\label{eq:gprime}
-c \frac{2}{1+x^2} \frac{\pi}{2} = 
-\frac{2}{\pi}\frac{1}{1+x^2}.
\end{eqnarray}
Note that $g(1) = 1/2.$  Indeed, one has
$$%\begin{eqnarray*}
\int_0^1 \frac{1}{1+r^2} \ln\Big{|}\frac{r+1}{r-1}\Big{|} \frac{dr}{r} 
=
\int_1^\infty \frac{r^2}{1+r^2} \ln\Big{|}\frac{r+1}{r-1}\Big{|} \frac{dr}{r} 
$$
Thus, with $c=2/\pi^2,$
\begin{eqnarray}
\nonumber
g(1) &=&
c \int_0^1 \frac{1}{1+r^2} \ln\Big{|}\frac{r+1}{r-1}\Big{|} \frac{dr}{r} 
+
c \int_1^\infty \frac{1}{1+r^2} \ln\Big{|}\frac{r+1}{r-1}\Big{|} \frac{dr}{r}  \\
&=&
\nonumber
c \int_1^\infty \frac{r^2}{1+r^2} \ln\Big{|}\frac{r+1}{r-1}\Big{|} \frac{dr}{r}  
+
c \int_1^\infty \frac{1}{1+r^2} \ln\Big{|}\frac{r+1}{r-1}\Big{|} \frac{dr}{r}  \\
&=&
\label{eq:g(1)}
c
\int_1^\infty \ln\Big{|}\frac{r+1}{r-1}\Big{|} \frac{dr}{r} 
=
1/2.
\end{eqnarray}
Then, from \eqref{eq:gprime} and \eqref{eq:g(1)} one
has
$$
g(x) = \frac{2}{\pi}\left(\frac{\pi}{2} - \arctan{x}\right)
= \frac{2}{\pi} \arctan(1/x)
$$
and the result follows setting $x=\sqrt{\theta}/a.$
\end{proof}

 \begin{cor}
\begin{equation*}
\lim_{n\to\infty}\prod_{j=1}^n\frac{\prod_{i=1}^jR_i^2}{ \theta+\prod_{i=1}^jR_i^2} \ \text{exists}, \quad \theta > 0.
%&& \lim_{n\to\infty}P(\min_{|s|=n}\sum_{j=1}^n |\mathbf{W}_{s|j}|^{-2}T^{(s|j)} \le  M) = 0 \quad \forall M > 0.
\end{equation*}
\end{cor}
 \begin{proof}
 The limit exists by virtue of being a positive super-martingale.
 \end{proof}  
 
\section*{Acknowledgments}   This work was partially supported by
grants DMS-1408947, DMS-1408939 and DMS-1211413
from the National Science Foundation.

%%%%%% BIBLIOGRAPHY %%%%%

\end{document}